\documentclass[12pt]{amsart}
\usepackage[utf8]{inputenc}
\usepackage[T1]{fontenc}
\usepackage{amsfonts}
\usepackage[leqno]{amsmath}
\usepackage{amssymb, comment, float}
\usepackage[dvipsnames]{xcolor}
\usepackage[foot]{amsaddr}
\usepackage[shortlabels]{enumitem}
\usepackage{mathdots}
\usepackage[a4paper, footskip=0.5in,  headheight = 0.5in, top=1.25in, bottom=1.25in,  right=1in,  left=1in]{geometry}
\usepackage{hyperref}
\hypersetup{
colorlinks,
linkcolor=black,
urlcolor=Blue,
citecolor=black,}
\usepackage[center]{caption}
\usepackage{eufrak}
\usepackage{marvosym}     
\usepackage{latexsym}
\usepackage[nodayofweek]{datetime}
\usepackage{tikz}
\usepackage{subfig}
\usepackage{thm-restate}

\newtheorem{theorem}{Theorem}[section]
\newtheorem{lemma}[theorem]{Lemma}
\newtheorem{conjecture}[theorem]{Conjecture}

\newtheorem{question}[theorem]{Question}
\newtheorem{observation}[theorem]{Observation}

\DeclareMathOperator{\tw}{tw}
 \DeclareMathOperator{\cone}{{\text{\sf{cone}}}}

\def\dd{\hbox{-}}   

\newcommand{\mf}{\mathfrak}

\newcounter{tbox}
\newcommand{\sta}[1]{\medskip\medskip\refstepcounter{tbox}\noindent{\parbox{\textwidth}{(\thetbox) \emph{#1}}}\vspace*{0.3cm}}
\newcommand{\mylongtitle}[1]{%
  \ifodd\value{page}%
    \protect\parbox{0.97\linewidth}{#1}\hfill%
  \else%
    \hfill\protect\parbox{0.97\linewidth}{#1}%
  \fi%
}

\title[Even-hole-free graphs of large treewidth.]{Chordal graphs, even-hole-free graphs and sparse obstructions to bounded treewidth}

\author{Sepehr Hajebi$^{\dagger}$}
\thanks{$^{\dagger}$ Department of Combinatorics and Optimization, University of Waterloo, Waterloo, ON, Canada.}
\date{\today}

\begin{document}
\maketitle

\begin{abstract}
Even-hole-free graphs pose a central challenge in identifying hereditary classes of bounded treewidth. We investigate this matter by presenting and studying the following conjecture: for an integer $t\geq 4$ and a graph $H$, every even-hole-free graph of large enough treewidth has an induced subgraph isomorphic to either $K_t$ or $H$, if (and only if) $H$ is a $K_4$-free chordal graph. The ``only if'' part follows from the properties of the so-called \textit{layered wheels}, a construction by Sintiari and Trotignon consisting of (even-hole, $K_4$)-free graphs with arbitrarily large treewidth.

Alecu, Chudnovsky, Spirkl and the author proved recently that the conjecture holds in two special cases: 
(a) when $t=4$; and (b) when $H=\cone (F)$ for some forest $F$; that is, $H$ is obtained from a forest $F$ by adding a universal vertex. Our first result is a common strengthening of (a) and (b): for an integer $t\geq 4$ and graphs $F$ and $H$, (even-hole, $\cone(\cone (F))$, $H$, $K_t$)-free graphs have bounded treewidth if and only if $F$ is a forest and $H$ is a $K_4$-free chordal graph.

Also, for general $t\geq 4$, we push the current state of the art further than (b) by settling the conjecture for the smallest choices of $H$ that are \textit{not} coned forests. The latter follows from our second result: we prove the conjecture when $H$ is a \textit{crystal}; that is, a graph obtained from arbitrarily many coned double stars by gluing them together along the ``middle'' edges of the double stars.

In the first version of this paper, we suggested the following which is a strengthening of our main conjecture: for every $t\geq 1$, every graph of sufficiently large treewidth has an induced subgraph of treewidth $t$ which is either complete, complete bipartite, or $2$-degenerate. This strengthening has now been refuted by Chudnovsky and Trotignon [On treewidth and maximum cliques, \href{https://arxiv.org/abs/2405.07471}{\texttt{arxiv:2405.07471}}, 2024].
\end{abstract}

\section{Introduction}\label{sec:intro}

\setcounter{subsection}{-1}

\subsection{Skipping pass for the well-versed}This paper features and studies the following conjecture:

\begin{conjecture}\label{conj:getgo}
    For all $t\geq 4$ and every $K_4$-free chordal graph $H$, every even-hole-free graph of sufficiently large treewidth has an induced subgraph isomorphic to $K_t$ or $H$.
\end{conjecture}
We have two main results: Theorems~\ref{thm:maincone1} and \ref{thm:maincrystal1}. The former unifies (the only) two preexisting partial results \cite{twxi} on Conjecture~\ref{conj:getgo} into a common strengthening, and the latter introduces new graphs $H$ satisfying the conjecture holds for all $t\geq 4$. 

 Broadly, this work is tied to an ongoing search for a characterization of bounded treewidth in graph classes defined by excluded induced subgraphs.  True or false, it appears to us that Conjecture~\ref{conj:getgo} isolates one of the main difficulties with finding the desired characterization. 
 
 However, this may not be appreciated immediately. We therefore feel compelled to take the next few subsections to delve into the details. A reader who is familiar with the topic, and shares the above perspective with us, or at least sees sufficient motivation behind the work, may jump straight ahead to Subsection~\ref{sub:main}.

\subsection{Background}Graphs in this paper have finite vertex sets, no loops and no parallel edges. The \textit{treewidth} of a graph $G$, denoted $\tw(G)$, is the smallest integer $w\geq 1$ for which there is a tree $T$ as well as a subtree of $T$ assigned to each vertex of $G$, such that the subtrees corresponding to adjacent vertices of $G$ intersect, and each vertex of $T$ belongs to at most $w+1$ subtrees corresponding to the vertices of $G$.

This may be viewed as a natural extension of the classical ``Helly property'' for subtrees \cite{helly}: \textit{given a tree $T$ and a collection of pairwise intersecting subtrees of $T$, some vertex of $T$ belongs to all subtrees}. The treewidth in turn measures how few subtrees each vertex of $T$ can be guaranteed to appear in, provided only \textit{some} pairs of subtrees intersect.

\begin{figure}[t!]
    \centering
    \includegraphics[scale=0.7]{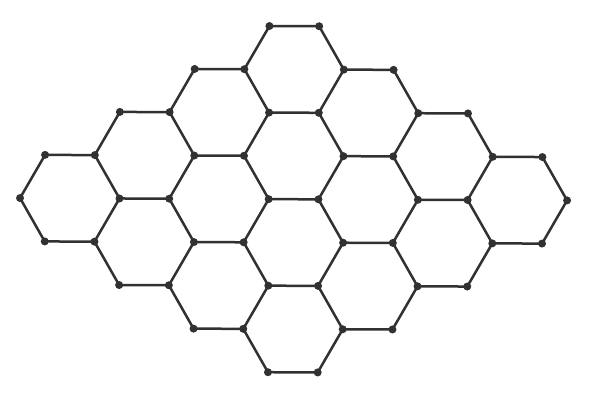}
    \caption{The graph $W_{5\times 5}$.}
    \label{fig:hexwall}
\end{figure}

Unlike the formal definition of the treewidth which we have nothing to do with, the present paper is concerned with what causes a graph to have large treewidth. This is a subject of great interest in structural and algorithmic graph theory because, once the graphs we study have bounded treewidth, there comes a variety of structural and algorithmic techniques to deal with them \cite{GMV, GMII}.

There are many examples of graphs with large treewidth, two of which are dense: the Helly property of subtrees implies that, for each $t\geq 1$, the complete graph $K_{t+1}$ and the complete bipartite graph $K_{t,t}$ both have treewidth $t$. All other examples are sparse (crucially, as sparse as possible; we will come back to this in Subsection~\ref{sub:2-deg}), among which the most famous one is the \textit{hexagonal grid}. For every integer $t\geq 1$, the $t$-by-$t$ hexagonal grid, also known as the $t$-by-$t$ \textit{wall}, denoted $W_{t\times t}$, has treewidth $t$ \cite{GMV} (see Figure~\ref{fig:hexwall}).
It is also well-known \cite{diestel} that subdividing the edges of a graph does not affect its treewidth, and so all subdivisions of $W_{t\times t}$ have treewidth $t$.

More importantly, there is a qualitative converse. In 1986, Robertson and Seymour \cite{GMV} proved the following still startling result: if a graph $G$ has large enough treewidth, then one may obtain from $G$, by a series of vertex- and/or edge-deletions, a subdivision of a (relatively) large wall. This is known as the ``Grid (Minor) Theorem,'' and remains one of the most celebrated highlights of the graph minors project:
\begin{theorem}[Robertson and Seymour \cite{GMV}] \label{thm:gridthm}
For every $t\geq 1$, every graph of sufficiently large treewidth contains a subdivision of $W_{t\times t}$ as a subgraph.
\end{theorem}
\begin{figure}[t!]
    \centering
    \includegraphics[scale=0.6]{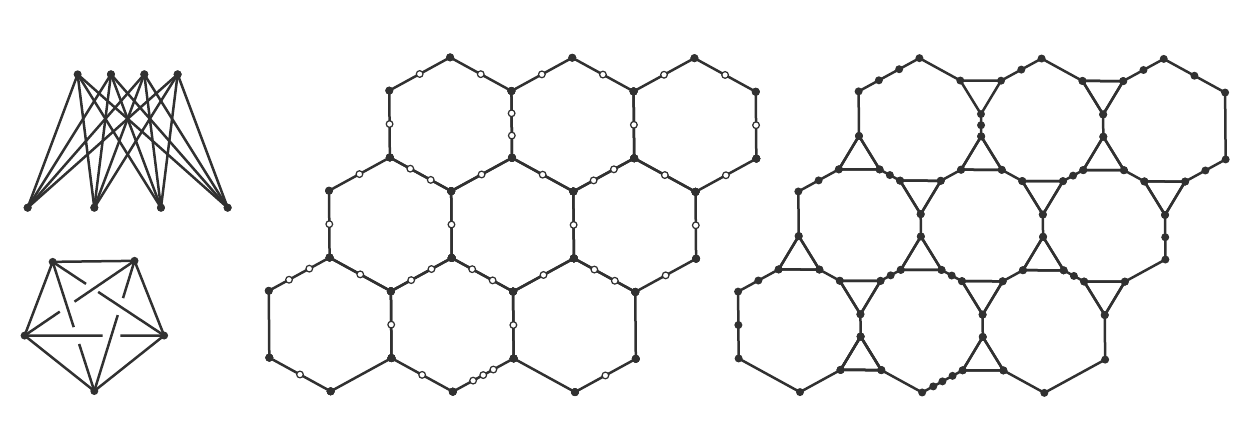}
    \caption{The $4$-basic obstructions (the graph on the right is the line graph of the subdivided wall in the middle).}
    \label{fig:basic}
\end{figure}

What if we only allow for removing vertices? Given a graph $G$ with vertex set $V(G)$ and edge set $E(G)$, an \textit{induced subgraph} of $G$ is the graph $G\setminus X$ for some $X\subseteq V(G)$, that is, the graph obtained from $G$ by removing the vertices in $X$. For $X\subseteq V(G)$, we use both $X$ and $G[X]$ to denote the subgraph of $G$ induced on $X$, which is the same as $G\setminus (V(G)\setminus X)$. 

An analog of Theorem~\ref{thm:gridthm} for induced subgraphs had not received much attention until a couple of years ago. However, some clues have long been around:  complete graphs, complete bipartite graphs and subdivided walls are three types of ``induced subgraph obstructions'' to bounded treewidth. This means there are graphs with arbitrarily large treewidth of each type, and given a graph $H$ from a certain type, all induced subgraphs of $H$ with large (enough) treewidth are also of the same type.

There is also a fourth type of obstruction: line graphs of subdivided walls. It is now standard terminology to refer to all these graphs together as the \textit{basic obstructions}. More exactly, for an integer $t\geq 1$, by a \textit{$t$-basic obstruction} we mean one of the following: the complete graph $K_{t+1}$, the complete bipartite graph $K_{t,t}$, subdivisions of $W_{t\times t}$ and the line-graphs of subdivisions of $W_{t\times t}$ (see Figure~\ref{fig:basic}). The $t$-basic obstructions all have treewidth $t$ \cite{GMV}.

\subsection{Why even-hole-free graphs?} For graphs $G$ and $H$, we say that $G$ \textit{contains} $H$ if $H$ is isomorphic to an induced subgraph of $G$; otherwise, we say $G$ is \textit{$H$-free}. For a family $\mathcal{H}$ of graphs, we say $G$ is $\mathcal{H}$-free if $G$ is $H$-free for all $H\in \mathcal{H}$. In pursuit of a ``grid theorem'' for induced subgraphs, our discussion from the previous subsection begs the following question: 
\begin{question}\label{q:basicall?}
    Is it true that for each $t\geq 1$, every graph of sufficiently large treewidth contains a $t$-basic obstruction?
\end{question}
\begin{figure}[t!]
    \centering
    \includegraphics[scale=0.5]{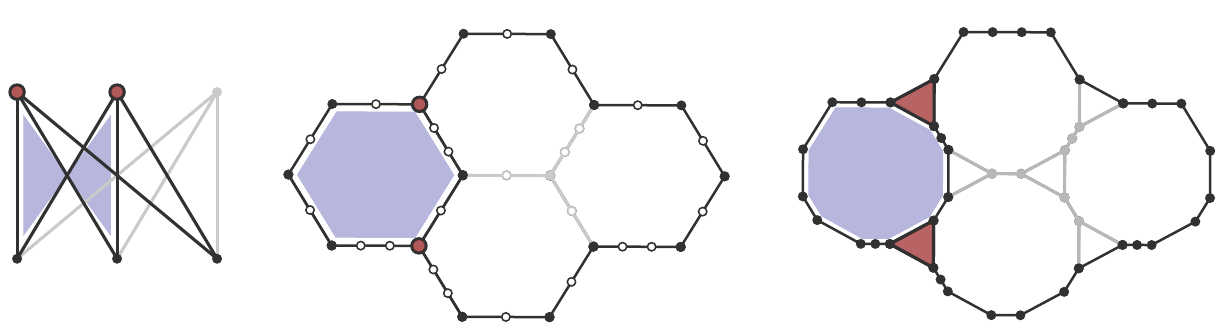}
    \caption{For $t\geq 3$, all non-complete $t$-basic obstructions have even holes. Note, in thick lines, the theta on the left and the middle, and the prism on the right, and note the even holes highlighted in them.}
    \label{fig:basic3pc}
\end{figure}
One may check that for $t\geq 3$, all $t$-basic obstructions except for complete graphs contain even holes (see Figure~\ref{fig:basic3pc}; recall that a \textit{hole} in a graph $G$ is an induced cycle on four or more vertices and an \textit{even hole} in $G$ is a hole on an even number of vertices. The definition of a ``theta'' and a ``prism'' will appear in Section~\ref{sec:2-tree}). Therefore, an even easier question to ask is:
\begin{question}\label{q:ehf}
    Is it true that for every $t\geq 1$, the class of all (even-hole, $K_t$)-free graphs has bounded treewidth?
\end{question}
In other words, as for approaching Question~\ref{q:basicall?}, inspired by the properties of the basic obstructions, it is natural to ask: \textit{do all obstructions (basic or not) except for complete graphs contain even holes?} This indicates the pivotal role of even-hole-free graphs in the interplay between the induced subgraphs and the treewidth (which, despite decades of extensive research on these graphs, seems to have not been anticipated for so long).

\subsection{The difficulty with even-hole-free graphs.} As said above, hoping for a non-trivial success toward the answer to Question~\ref{q:basicall?}, Question~\ref{q:ehf} is a natural weakening to begin with. The case $t\leq 3$ has already been solved \cite{evenholetrianglefree}: (even-hole, $K_3$)-free graphs have treewidth at most five. In 2019, Sintiari and Trotignon \cite{layered-wheels} answered Question~\ref{q:ehf} with an emphatic no. They provided explicit examples of (even-hole, $K_4$)-free graphs with arbitrarily large treewidth called the \textit{layered wheels}. 

There are two important observations about the layered wheels. First, although this will soon be seen to be a red herring, the full definition of these graphs \cite{layered-wheels} is quite technical, almost contrived-looking (it takes a couple of pages, which is why we omit it). Second, and more importantly in the context of this paper, unlike the global structure of layered wheels, their local structure is highly restricted; indeed, as restricted as possible. To explain this, we need the exact statement of the main result of \cite{layered-wheels}:

\begin{theorem}[Sintiari and Trotignon \cite{layered-wheels}]\label{thm:layeredwheels}
    For every integer $h\geq 1$, there are (even-hole, $K_4$)-free graphs of arbitrarily large treewidth in which there is no induced subgraph on $h$ or fewer vertices containing a hole.
\end{theorem}

Graphs with no hole are called \textit{chordal}. It is well-known that among graphs of any given clique number, chordal graphs of the same clique number are exactly the 
graphs with the smallest possible treewidth (up to taking induced subgraphs). Consequently, by Theorem~\ref{thm:layeredwheels}, layered wheels are (even-hole, $K_4$)-free graphs of arbitrarily large treewidth in which the treewidth of every small induced subgraph -- regardless of how large our notion of ``small'' is -- is always the smallest possible. 

\subsection{The conjecture and the results (old and new)}\label{sub:main} The above observation leads us to a strange contrast between the global and the local structure of the layered wheels. We believe this contrast addresses the very nature of (not only the layered wheels but) all even-hole-free graphs of large treewidth (and bounded clique number).

Let us elaborate. For an integer $k\geq 1$, by a \textit{$k$-forest} we mean a $K_{k+2}$-free chordal graph. For instance, $1$-forests are exactly the forests. Now, suppose $H$ is a graph such that every layered wheel of sufficiently large treewidth contains $H$. Then, by Theorem~\ref{thm:layeredwheels}, $H$ must be a $2$-forest. Furthermore, as one may check, using the construction of layered wheels \cite{layered-wheels}, the converse is true, as well:
\begin{observation}\label{obs:chordal}
   Given a graph $H$, every layered wheel of sufficiently large treewidth contains $H$ if and only if $H$ is a $2$-forest. 
\end{observation}

Indeed, the ``only if'' implication in Observation~\ref{obs:chordal} extends beyond layered wheels to all (even-hole, $K_4$)-free graphs, and even further to larger upper bounds on the clique number. Explicitly, let $t\geq 4$ be an integer and let $H$ be a graph such that every (even-hole, $K_t$)-free graph of sufficiently large treewidth contains $H$. Then, it follows from Theorem~\ref{thm:layeredwheels} that $H$ is a $2$-forest (again, there is no restriction on $H$ for $t\leq 3$ because (even-hole, $K_3$)-free graphs have treewidth at most five \cite{evenholetrianglefree}).

We conjecture that, like the situation with Observation~\ref{obs:chordal}, the converse is also true:
\begin{conjecture}\label{conj:mainiff}
    Given an integer $t\geq 4$ and a $2$-forest $H$, every (even-hole, $K_t$)-free graph of sufficiently large treewidth contains $H$ if and only if $H$ is a $2$-forest.
\end{conjecture}

In other words, Conjecture~\ref{conj:mainiff} asserts that, as opposed to their global complexity, layered wheels are quite canonical from a local point of view, in that they represent the local structure of \textit{all} even-hole-free graphs of large treewidth and small clique number.

Once again, the ``only if'' implication in Conjecture~\ref{conj:mainiff} follows from Theorem~\ref{thm:layeredwheels}. Our main question is whether the ``if'' implication is true (which is identical to Conjecture~\ref{conj:getgo} we started the paper with):

\begin{conjecture}\label{conj:main}
For every integer $t\geq 4$ and every $2$-forest $H$,  (even-hole, $H$, $K_t$)-free graphs have bounded treewidth.
\end{conjecture}

There are only two partial results on Conjecture~\ref{conj:main}, recently established by Alecu, Chudnovsky, Spirkl and the author \cite{twxi}. The first one settles the case $t=4$:
\begin{theorem}[Alecu, Chudnovsky, Hajebi and Spirkl \cite{twxi}]\label{thm:tw11k4} For every $2$-forest $H$, every (even-hole, $K_4$)-free graph of sufficiently large treewidth contains $H$.
\end{theorem}

Incidentally, Theorem~\ref{thm:tw11k4} contains a conjecture of Sintiari and Trotignon \cite{layered-wheels} as a (very) special case: when $H$ is obtained from a two-edge path by adding a universal vertex. The latter graph $H$ is commonly known as the \textit{diamond}. In general, given a graph $F$, we denote by $\cone(F)$ the graph obtained from $F$ by adding a universal vertex (see Figure~\ref{fig:2-treeassorted}).

\begin{figure}
    \centering
    \includegraphics[scale=0.6]{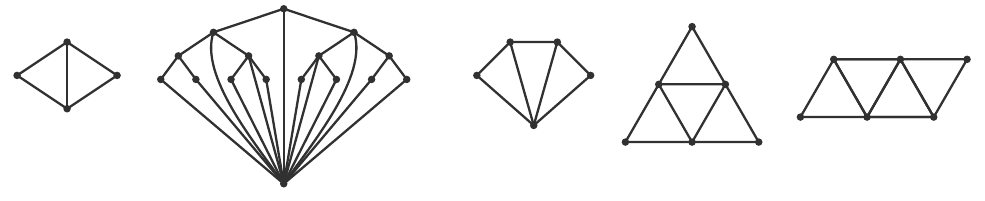}
\caption{From left to right: the diamond, a coned tree, the gem, and the two smallest  ($2$-connected) $2$-forests that are not coned forests (the last one but one is commonly known as the \textit{antinet}).}
    \label{fig:2-treeassorted}
\end{figure}

The second result in \cite{twxi}, incomparable to Theorem~\ref{thm:tw11k4}, is another strengthening of Sintiari and Trotignon's conjecture, that Conjecture~\ref{conj:main} holds if $H=\cone(F)$ for a forest $F$. Note that this is a characterization of forests, because $\cone(F)$ is a $2$-forest if and only if $F$ is a forest:
\begin{theorem}[Alecu, Chudnovsky, Hajebi and Spirkl \cite{twxi}]\label{thm:tw11conedforest} For every integer $t\geq 4$ and every forest $F$, (even-hole, $\cone(F), K_t$)-free graphs have bounded treewidth.
\end{theorem}

Our first result in this paper is a common strengthening of Theorems~\ref{thm:tw11k4} and \ref{thm:tw11conedforest} (the proof of which does not rely on those theorems).

\begin{theorem}\label{thm:maincone1}
   Let $t\geq 1$ be an integer, let $F$ be a forest and let $H$ be a $2$-forest. Then (even-hole, $\cone(\cone (F))$, $H$, $K_t$)-free graphs have bounded treewidth. 
\end{theorem}
We remark that for each $t\geq 4$, in order for Theorem~\ref{thm:maincone1} to hold true, $F$ must be a forest and $H$ must be a $2$-forest.

As a different approach, one may also improve on each of Theorems~\ref{thm:tw11k4} and \ref{thm:tw11conedforest} separately. The one we do have something to contribute to is extending Theorem~\ref{thm:tw11conedforest} beyond coned forests. Said differently, we wish to identify more graphs $H$ for which Conjecture~\ref{conj:main} holds for all $t\geq 4$. To begin with, it is straightforward to check that among $2$-forests which are not coned forests, there are only two (up to isomorphism and $2$-connectivity) with the fewest number of vertices.

Both these graphs are obtained from a coned $K_2$, that is, a triangle, and a coned three-edge path, commonly known as a \textit{gem} (which is annoyingly more similar to a ``diamond'' than the diamond), by gluing them together along the unique edge of $K_2$ and some edge of the three-edge path (See Figure~\ref{fig:2-treeassorted}).

We prove that for all $t\geq 4$, Conjecture~\ref{conj:main} holds for both of these graphs as a choice of $H$. To be more accurate, we prove Theorem~\ref{thm:maincrystal1} below which is a much stronger result. A \textit{double star} is a tree with exactly two vertices of degree more than one (which are necessarily adjacent). The \textit{middle edge} of a double star is the unique edge between the two vertices of degree more than one. By a \textit{crystal} we mean a graph obtained from $\cone(F_1),\ldots, \cone(F_k)$ for some integer $k\geq 1$ and some choice of $k$ double stars $F_1,\ldots, F_k$, by identifying the middle edges of $F_1,\ldots, F_k$ (see Figure~\ref{fig:crystal}).

Our second result in this paper is the following:
\begin{theorem}\label{thm:maincrystal1}
   Let $t\geq 1$ be an integer and let $H$ be a crystal. Then, (even-hole, $H$, $K_t$)-free graphs have bounded treewidth. 
\end{theorem}
\begin{figure}
    \centering
    \includegraphics[scale=0.5]{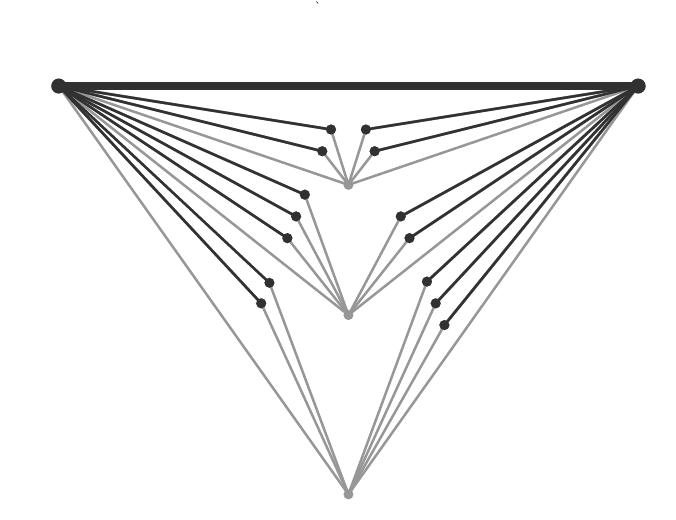}
    \caption{A crystal.}
    \label{fig:crystal}
\end{figure}
\subsection{The connection with sparse obstructions}\label{sub:2-deg}
Another approach would be to generalize Theorem~\ref{thm:tw11k4} from (even-hole, $K_4$)-free graphs to (even-hole, $K_t$)-free graphs for some large values of $t$. We do not have much to say about this, except it appears that we are somehow bound to prove Conjecture~\ref{conj:main} for all values of $t>4$ at once. This leads to the following conjecture, which, if true, will imply Conjecture~\ref{conj:main} by reducing it to Theorem~\ref{thm:tw11k4}:

\begin{conjecture}\label{conj:ehfk4}
    For every $t\geq 1$, every even-hole-free graph of large enough treewidth has an induced subgraph of treewidth $t$ which is either complete or $K_4$-free.
\end{conjecture}

Furthermore, apart from the layered wheels and the basic obstructions, there are now many other constructions of graphs with arbitrarily large treewidth \cite{deathstar, Davies2}, and they all happen to be $K_4$-free. As far as we can tell, Conjecture~\ref{conj:ehfk4} may also be true outside the class of even-hole-free graphs: 

\begin{conjecture}\label{conj:generalk4}
    For every $t\geq 1$, every graph of large enough treewidth has an induced subgraph of treewidth $t$ which is either complete or $K_4$-free.
\end{conjecture}

 There might be even more to say: among the basic obstructions and the non-basic ones that have been discovered so far \cite{deathstar,Davies2,layered-wheels}, apart from complete graphs and complete bipartite graphs, the rest are all $2$-degenerate (recall that a graph $G$ is \textit{$d$-degenerate} for an integer $d\geq 0$ if every induced subgraph of $G$ has a vertex of degree at most $d$). We point out that this is as radical of a sparsity dichotomy as possible because $1$-degenerate graphs are forests, which have treewidth $1$.
 
 Accordingly, we propose the following strengthening of Conjecture~\ref{conj:generalk4}:

\begin{conjecture}\label{conj:2-deg}
    For every integer $t\geq 1$, every graph of large enough treewidth has an induced subgraph of treewidth $t$ which is either complete, complete bipartite, or $2$-degenerate.
\end{conjecture}

This leaves us with four successively stronger conjectures, namely \ref{conj:main}, \ref{conj:ehfk4}, \ref{conj:generalk4} and \ref{conj:2-deg}. We tend to believe Conjecture~\ref{conj:main}, however, we have virtually nothing to contribute to Conjecture~\ref{conj:2-deg}. A counterexample to the former or a proof of the latter would be of great interest.\\

\textbf{Note.} Since the earliest appearance of this paper, Conjecture~\ref{conj:generalk4} was disproved by Chudnovsky and Trotignon \cite{newlayeredwheel} using a new variant of the layered wheels. This refutes Conjecture~\ref{conj:2-deg} as well. Their counterexample is far from even-hole-free and contains all chordal graphs as induced subgraphs. This leaves open our main conjecture, namely Conjecture~\ref{conj:main}. Conjecture~\ref{conj:ehfk4} too remains open.

\section{$2$-trees and other preliminaries}\label{sec:2-tree}

In this section, we will reduce our main results, Theorems~\ref{thm:maincone1} and \ref{thm:maincrystal1}, to Theorems~\ref{thm:maincone2} and \ref{thm:maincrystal2}, respectively. The latter two better fit the technical circumstances of our proofs. 

We start with some definitions. Let $G$ be a graph and let $x\in V(G)$. We denote by $N_G(x)$ the set of all neighbors of $x$ in $G$, and write $N_G[x]=N_G(x)\cup \{x\}$. For an induced subgraph $H$ of $G$ (not necessarily containing $x$), we define $N_H(x)=N_G(x) \cap H$ and $N_H[x]=N_H(x)\cup \{x\}$. Also, for $X\subseteq G$, we denote by $N_G(X)$ the set of all vertices in $G\setminus X$ with at least one neighbor in $X$, and define $N_G[X]=N_G(X)\cup X$. Let $X, Y \subseteq G$ be disjoint. We say $X$ is \textit{complete} to $Y$ if every vertex in $X$ is adjacent to every vertex in $Y$ in $G$, and $X$ is \emph{anticomplete}
to $Y$ if there is no edge in $G$ with an end in $X$ and an end in $Y$. By a {\em path in $G$} we mean an induced subgraph $P$ of $G$ which is a path, and we write $P = p_1 \dd \cdots \dd p_k$ to mean $V(P) = \{p_1, \dots, p_k\}$ and $p_i$ is adjacent to $p_j$ if and only if $|i-j| = 1$. The vertices $p_1$ and $p_k$ are called the \emph{ends of $P$}, and we say that $P$ is \emph{between $p_1$ and $p_k$} or \emph{from $p_1$ to $p_k$}. The \emph{interior of $P$}, denoted by $P^*$, is the set $P \setminus \{p_1, p_k\}$. The \emph{length} of a path is its number of edges (so a path of length at most one has an empty interior).
\begin{figure}[t!]
\centering
\includegraphics[scale=0.7]{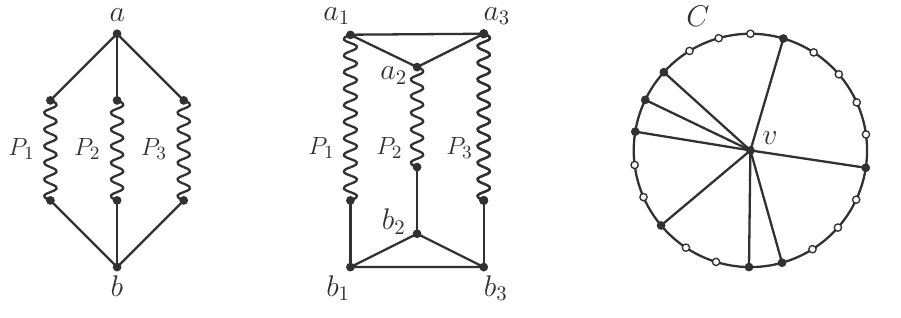}
\caption{From left to right, a theta, a prism and an even wheel. Squiggly lines represent paths of arbitrary (possibly zero) lengths.}
\label{fig:forbidden_isgs}
\end{figure}

A {\em theta} is a graph consisting of two non-adjacent vertices and three paths of length at least two between them whose interiors are pairwise disjoint and anticomplete. The line graph of a theta is called a \textit{prism}. An {\em even wheel} is a graph consisting of a cycle on four or more vertices and a vertex outside the cycle with at least four, and even number of, neighbors in the cycle (see Figure~\ref{fig:forbidden_isgs}). It is straightforward to check that the class of even-hole-free graph is contained in the class $\mathcal{E}$ of ($K_{2,2}$, theta, prism, even-wheel)-free graphs. That said, all our results are proved in the extended class $\mathcal{E}_t$ of ($K_{2,2}$, theta, prism, even-wheel, $K_t$)-free of graphs.

For an integer $n$, we write $[n]$ for the set of all positive integers no larger than $n$ (so we have $[n]=\emptyset$ if $n\leq 0$). For $c\geq 0$, by a \textit{$c$-clique} in a graph $G$ we mean a clique in $G$ of cardinality $c$. Let $k\geq 1$ be an integer. A \textit{$k$-tree} is a graph $\nabla$ which is either a $k$-vertex complete graph, or we have $|V(\nabla)|=h>k$ and there exists a bijection $\varpi_{\nabla}:[h]\rightarrow V(\nabla)$ such that for every $i\in [h-k]$, the set of neighbors of $\varpi_{\nabla}(i)$ in $V(\nabla)\setminus \varpi_{\nabla}([i])$ is a $k$-clique in $\nabla$. It is straightforward to observe that every $k$-tree is a $k$-forest (recall that a $k$-forest means a $K_{k+2}$-free chordal graph). Conversely, every $k$-forest is an induced subgraph of a $k$-tree:
\begin{theorem}[Alecu, Chudnovsky, Hajebi and Spirkl; see Theorem 1.9 in \cite{twxi}]\label{thm:k-tree}
    For every $k\geq 1$ and every $k$-forest $H$, there exists a $k$-tree $\nabla$ such that $H$ is an induced subgraph of $\nabla$.
\end{theorem}

From Theorem~\ref{thm:k-tree} for $k=1,2$, it follows that every forest is contained in a tree, and every $2$-forest is contained in a $2$-tree. This enables us to prove Theorem~\ref{thm:maincone1} only when $F$ is a tree and $H$ is $2$-tree.

Moreover, we can allow $F$ to be ``regular'' and not necessarily induced. This calls for another definition: for integers $d\geq 1$ and $r\geq 0$, let $T_{d,r}$ be the rooted tree of radius $r$ such that if $r\geq 1$, then the root has degree $d$, and every vertex that is neither the root nor a leaf has degree $d+1$ (see Figure~\ref{fig:T{d,r}}). It is easy to observe that $T_{d,r}$ contains every tree of maximum degree $d$ and radius $r$. There is also a well-known result of Kierstead and Penrice \cite{KP}, that in sparse graphs, we may pass from a non-induced tree to an induced one:

\begin{theorem}[Kierstead and Penrice \cite{KP}]\label{thm:kp} For all integers $d,s,t\geq 1$ and $r\geq 0$, there is an integer $f=f(d,r,s,t)\geq 1$ with the following property. Let $G$ be a $(K_{s,s},K_t)$-free graph and $U$ be a subgraph of $G$ which is isomorphic to $T_{f,f}$. Then there exists an induced subgraph $U'$ of $G$ with $V(U')\subseteq V(U)$ such that $U'$ is isomorphic to $T_{d,r}$.    
\end{theorem}
\begin{figure}
    \centering
    \includegraphics[scale=0.6]{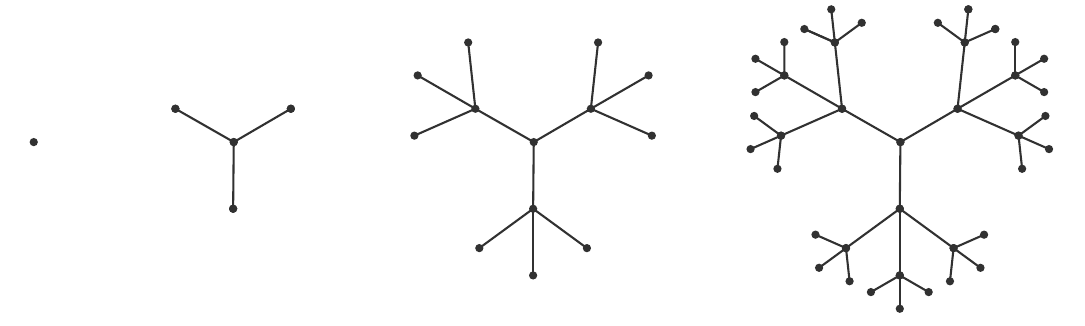}
    \caption{From left to right: the trees $T_{3,r}$ for $r=0,1,2,3$.}
    \label{fig:T{d,r}}
\end{figure}
In view of Theorems~\ref{thm:k-tree} and \ref{thm:kp}, in order to prove Theorem~\ref{thm:maincone1}, it suffices to show that:

\begin{theorem}\label{thm:maincone2}
    For all integers $f,t,h\geq 1$, there exists an integer $\tau_1=\tau_1(f,h,t)\geq 1$ such that for every graph $G\in \mathcal{E}_t$ of treewidth more than $\tau_1$, either $G$ contains every $2$-tree on $h+1$ vertices, or $G$ has a subgraph isomorphic to $\cone(\cone(T_{f,f}))$.
\end{theorem}

Let us deduce Theorem~\ref{thm:maincone1} from Theorem~\ref{thm:maincone2} right away:

\begin{proof}[Proof of Theorem~\ref{thm:maincone1} assuming 
Theorem~\ref{thm:maincone2}] Let $t\geq 1$ be an integer, let $F$ be a forest and let $H$ be a $2$-forest. By Theorem~\ref{thm:k-tree}, there is a tree $T$ which contains $F$ and there is a $2$-tree $\nabla$ which contains $H$. Let $d\geq 1$ and $r\geq 0$ be integers such that $T$ has maximum degree at most $d$ and radius at most $r$. Then $T_{d,r}$ contains $T$, which in turn contains $F$. Let $f=f(d,r,2,t)$ be as in Theorem~\ref{thm:kp}, let $h\geq 1$ such that $|V(\nabla)|\leq h+1$ and let $\tau_1=\tau_1(f,h,t)$ be as in Theorem~\ref{thm:maincone2}. 

Let $G$ be an (even-hole, $\cone(\cone(F))$, $H$, $K_t$)-free graph. Since $G$ is (even-hole, $K_t$)-free, it follows that $G\in \mathcal{E}_t$. Since $G$ is $H$-free, it follows that $G$ is $\nabla$-free, as well. Also, since $G$ is $(K_{2,2},\cone(\cone(F)),K_t)$-free, it follows from Theorem~\ref{thm:kp} that $G$ has no subgraph isomorphic to $\cone(\cone(T_{f,f}))$. Now, by Theorem~\ref{thm:maincone2}, $G$ has treewidth at most $\tau_1$, as required.
\end{proof}

Similarly, Theorem~\ref{thm:maincrystal1} can be reduced to a statement involving crystals that are non-induced but in a ``controlled'' manner. Like trees, it is also convenient to work with a ``regular'' version of crystals. Let $f,g\geq 1$ be integers, let $G$ be a graph and let $z_1,z_2\in V(G)$ be adjacent. A \textit{$(z_1,z_2,f,g)$-crystal in $G$} is a tuple $\mf{c}=(S_{1,z},z,S_{2,z}:z\in S)$ with the following specifications.

\begin{enumerate}[(CR1), leftmargin=15mm, rightmargin=7mm]
\item\label{CR1} $S$ is an $f$-subset of $V
(G)\setminus \{z_1,z_2\}$ (and thus $\mf{c}$ is a $3f$-tuple).
\item\label{CR2} \{$S_{1,z},S_{2,z}:z\in S\}$ are $2f$ pairwise disjoint $g$-subsets of $V(G)\setminus (S\cup \{z_1,z_2\})$.
\item\label{CR3}For each $i\in \{1,2\}$, every $z\in S$ and every $x\in S_{i,z}$, we have $N_{\{z_1,z_2,z\}}(x)=\{z_i,z\}$.
\end{enumerate}
See Figure~\ref{fig:regular_crystal}.
\begin{figure}
    \centering
    \includegraphics[scale=0.7]{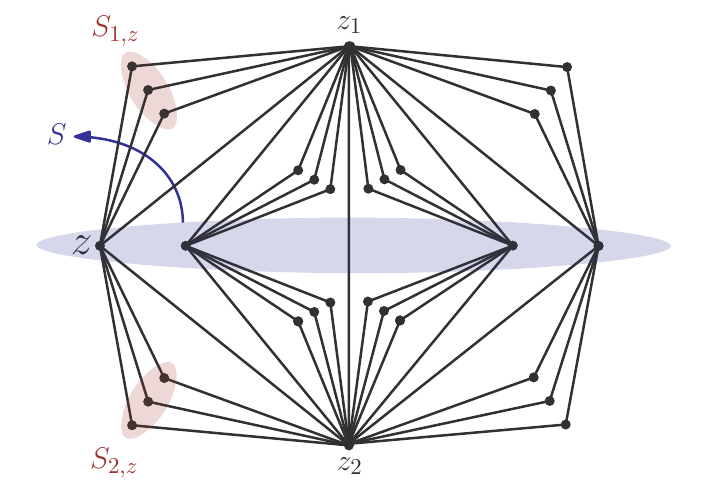}
    \caption{A $(z_1,z_2,4,3)$-crystal}
    \label{fig:regular_crystal}
\end{figure}
We write $V(\mf{c})=S\cup (\bigcup_{z\in S}(S_{1,z}\cup S_{2,z}))$. Also, by an \textit{$(f,g)$-crystal in $G$}, we mean a $(z_1,z_2,f,g)$-crystal in $G$ for some pair $z_1,z_2$ of adjacent vertices in $G$.

Note that a crystal $\mf{c}$ in $G$ may or may not spot an induced subgraph of $G$ which is a crystal. But if it does, then we say $\mf{c}$ is ``clear.'' More precisely, we say a $(z_1,z_2,f,g)$-crystal $\mf{c}=(S_{1,z},z,S_{2,z}:z\in S)$ in $G$ is \textit{clear} if $S$ is a stable set and $\{S_{1,z},S_{2,z}:z\in S\}$ are pairwise anticomplete stable sets in $G$. Just like Theorem~\ref{thm:kp}, we need to pass, in sparse graphs, from crystals to clear crystals. This is attained in Theorem~\ref{thm:clearingacrystal} below, using the next two classical results from Ramsey Theory:

\begin{theorem}[Ramsey \cite{multiramsey}]\label{classicalramsey}
For all integers $c,s\geq 1$, every graph $G$ on at least $c^s$ vertices contains either a $c$-clique or a stable set of cardinality $s$. 
\end{theorem}

\begin{theorem}[Folklore; see Lemma 2.5 in \cite{pinned}]\label{ramsey2}
For all integers $q,r,s,t\geq 1$, there exists an integer $n=n(q,r,s,t)\geq 1$ with the following property. Let $G$ be a $(K_{s,s},K_t)$-free graph. Let $\mathcal{X}$ be a collection of pairwise disjoint subsets of $V(G)$, each of cardinality at most $r$, with $|\mathcal{X}|\geq n$. Then there are $q$ distinct sets $X_1,\ldots, X_q\in \mathcal{X}$ which are pairwise anticomplete in $G$.
\end{theorem}

\begin{theorem}\label{thm:clearingacrystal}Let $f,g,s,t\geq 1$ be integers and let $n(\cdot,\cdot,\cdot,\cdot)$ be the function as in Theorem~\ref{ramsey2}. Let $n_1=n_1(f,g,s,t)=n(f,2g+1,s,t)\geq 1$ and let $n_2=n_2(f,g,s,t)=n(2g,2,s,t)\geq 1$. Let $G$ be a $(K_{s,s},K_{t})$-free graph let $z_1,z_2\in V(G)$ be adjacent, and let $\mf{c}=(\Lambda_{1,z},z,\Lambda_{2,z}:z\in \Lambda)$ be a $(z_1,z_2,n_1,n_2)$-crystal in $G$. Then there is a clear $(z_1,z_2,f,g)$-crystal in $G$.
\end{theorem}
\begin{proof}
   We first show that:

   \sta{\label{st:antistable} For each $i\in \{1,2\}$ and every $z\in \Lambda$, there exists a stable set $S_{i,z}\subseteq \Lambda_{i,z}$ of cardinality $g$ such that $S_{1,z}$ and $S_{2,z}$ are anticomplete in $G$.}

   Let $z\in \Lambda$ be fixed. Since $|\Lambda_{1,z}|=|\Lambda_{2,z}|=n_1$, we may choose a bijection $u:\Lambda_{1,z}\rightarrow \Lambda_{2,z}$. Consider the collection $\mathcal{X}=\{\{x,u(x)\}:x\in \Lambda_{1,z}\}$; then $|\mathcal{X}|=n_1$ and the elements of $\mathcal{X}$ are pairwise disjoint $2$-subsets of $V(G)$. Since $G$ is $(K_{s,s},K_t)$-free, it follows from the choice of $n_1$ and Theorem~\ref{ramsey2} that there are $2g$ vertices $x^1_1\ldots,x^1_g,x^2_1\ldots, x^2_g\in \Lambda_{1,z}$ for which the sets $\{\{x^i_j,u(x^i_j)\}:i\in \{1,2\}, j\in [g]\}$ are pairwise anticomplete in $G$. Let $S_{1,z}=\{x^1_j:j\in [g]\}\subseteq \Lambda_{1,z}$ and let $S_{2,z}=\{u(x^2_j):j\in [g]\}\subseteq \Lambda_{2,z}$. Then $S_{1,z}$ and $S_{2,z}$ are anticomplete stable sets in $G$, each of cardinality $g$. This proves \eqref{st:antistable}.
\medskip

Now, consider the collection $\mathcal{X}=\{S_{1,z}\cup S_{2,z}\cup \{z\}:z\in \Lambda\}$. Then $|\mathcal{X}|=n_2$ and the elements of $\mathcal{X}$ are pairwise disjoint $(2g+1)$-subset of $V(G)$. Since $G$ is $(K_{s,s},K_t)$-free, it follows from the choice of $n_2$ and Theorem~\ref{ramsey2} that there exists an $f$-subset $S$ of $\Lambda$ for which the set $\{S_{1,z}\cup S_{2,z}\cup \{z\}:z\in S\}$ are pairwise anticomplete in $G$. In particular, $S$ is a stable set, and from \eqref{st:antistable}, it follows that  $\{S_{1,z},S_{2,z}:z\in S\}$ are pairwise anticomplete stable sets in $G$. Hence, $\mf{c}'=(S_{1,z},z,S_{2,z}:z\in S)$ is a clear $(z_1,z_2,f,g)$-crystal in $G$. This completes the proof of Theorem~\ref{thm:clearingacrystal}.
\end{proof}

With Theorem~\ref{thm:clearingacrystal} in our hands, we may now reduce Theorem~\ref{thm:maincrystal1} to:

\begin{theorem}\label{thm:maincrystal2}
    For all integers $f,g,t\geq 1$, there exists an integer $\tau_2=\tau_2(f,g,t)\geq 1$ such that for every graph $G\in \mathcal{E}_t$ of treewidth ore  $\tau_2$, there is an $(f,g)$-crystal in $G$.
\end{theorem}

This only takes a brief argument:

\begin{proof}[Proof of Theorem~\ref{thm:maincrystal1} assuming 
Theorem~\ref{thm:maincrystal2}] Let $H$ be a crystal. Then there are integers $f,g\geq 1$ as well as $f$ double stars $H_1,\ldots, H_f$, each of maximum degree at most $g$, such that $H$ is obtained from $\cone(H_1),\ldots, \cone(H_f)$ by an identification of their middle edges. Let $n_1=n_1(f,g,2,t)\geq 1$ and $n_2=n_2(f,g,2,t)\geq 1$ be as in Theorem~\ref{thm:clearingacrystal}, and let $\tau_2=\tau_2(n_1,n_2,t)\geq 1$ be as in Theorem~\ref{thm:maincrystal2}. We claim that every $H$-free graph $G\in \mathcal{E}_t$ has treewidth at most $\tau_2$. To see this, observe that since $G$ is $H$-free, it follows that there is no clear $(f,g)$-crystal in $G$. This, combined with Theorem~\ref{thm:clearingacrystal} and the assumption that $G$ is $(K_{2,2},K_t)$-free, implies that there is no $(n_1,n_2)$-crystal in $G$. But now the claim follows from Theorem~\ref{thm:maincrystal2}. 
\end{proof}
Hence, it remains to prove Theorems~\ref{thm:maincone2} and \ref{thm:maincrystal2}.

\section{Phantom layered wheels}\label{sec:phantomintro}

Roughly, our proofs of Theorems~\ref{thm:maincone2} and \ref{thm:maincrystal2} are in two steps. First, we show that for every $t\geq 1$, every graph $G\in \mathcal{E}_t$ with sufficiently large treewidth has an induced subgraph $P$ which is an ``approximate version'' of a layered wheel of large treewidth. We call $P$ a ``phantom'' in $G$. The second step, then, is to take the phantom on a roller coaster ride of Ramsey-type arguments, extracting more and more similarities with the structure of the layered wheels until the desired result follows.

We complete these two steps in the reverse order in Sections~\ref{sec:phantom} and \ref{sec:conjuring}. Our goal in this section is to show how the above strategy leads to the proofs of Theorems~\ref{thm:maincone2} and \ref{thm:maincrystal2}.

First, we must clarify what we mean by the ``approximate'' version of the layered wheels. This needs a slightly clearer picture of the structure of these graphs \cite{layered-wheels}. Specifically, for an integer $r\geq 0$, a layered wheel $L$ of treewidth (at least) $r$ is a $K_{r+1}$-minor model in which each branching set induces a path. This means that $V(L)$ can be partitioned into $r+1$ pairwise disjoint paths $L_0,\ldots, L_r$ in $L$ such that for all distinct $i,j\in \{0,\ldots, r\}$, there is an edge of $L$ with an end in $L_i$ and an end in $L_j$. In fact, a lot more happens: \textit{for each $i\in [r]$, every two vertices $z_1,z_2\in V(L_0)\cup \cdots \cup V(L_{i-1})$ which are adjacent in $L$ have (many) ``private'' common neighbors in $L_{i}$.} This is crucial for the layered wheels to be even-hole-free (the underlying reason can be traced back to Theorem~\ref{thm:evenwheeltheta}), and our ``approximate'' layered wheels are exactly meant to capture this property.

We now give the formal definition. Let $G$ be a graph, let $Z_0\subseteq V(G)$ and let $d\geq 1$ and $r\geq 0$ be integers. A \textit{$(Z_0,d,r)$-phantom in $G$} is a $(2r+1)$-tuple $\mf{p}=(Z_0,\ldots, Z_r;\Gamma_i:i\in [r])$ with the following specifications.
\begin{enumerate}[(P1), leftmargin=15mm, rightmargin=7mm]
\item\label{P1} We have $Z_0\subseteq \cdots\subseteq Z_{r}\subseteq V(G)$.
    \item\label{P2} For every $i\in [r]$, $\Gamma_i$ is a map with domain $E(G[Z_{i-1}])$ such that:
    \begin{itemize}
        \item for each edge $e\in E(G[Z_{i-1}])$, we have $\Gamma_i(e)\subseteq Z_i\setminus Z_{i-1}$ with $|\Gamma_i(e)|= d$ and the ends of $e$ are complete to $\Gamma_i(e)$; and
        \item for all distinct $e,e'\in E(G[Z_{i-1}])$, we have $\Gamma_i(e)\cap \Gamma_i(e')=\emptyset$.
    \end{itemize}
\end{enumerate}

\begin{figure}[t!]
    \centering
\includegraphics[scale=0.6]{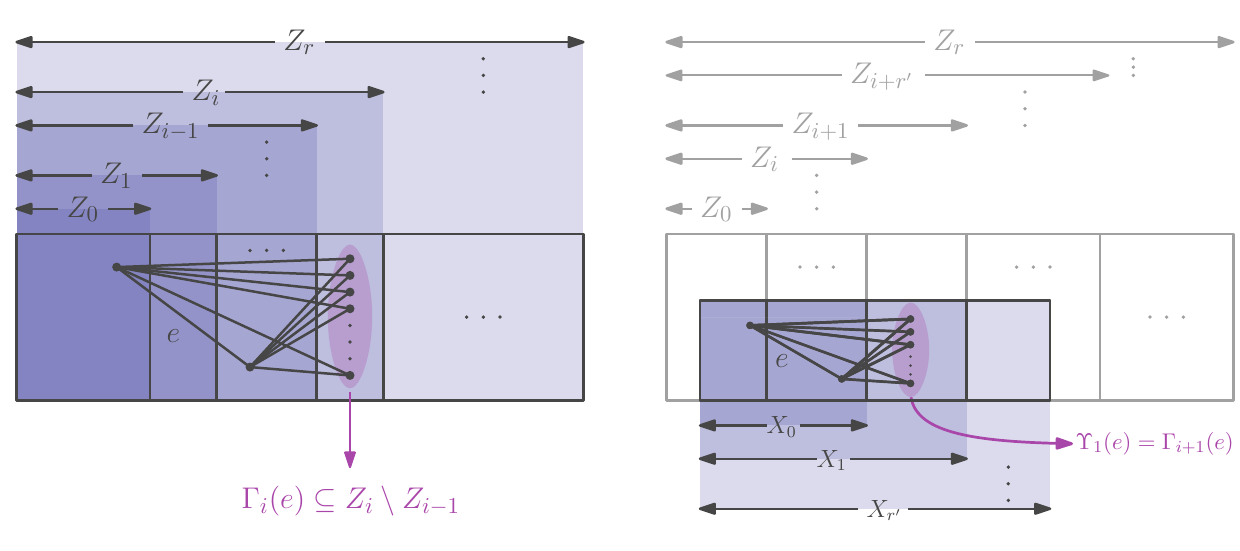}
\caption{Left: A $(Z_0,d,r)$-phantom $\mf{p}$. Right: The $(X_0,d,r')$-phantom $\mf{p}[X_0;i,r']$ for $X_0\subseteq Z_i$.}
    \label{fig:phantom}
\end{figure}
See Figure~\ref{fig:phantom}. 

Moreover, we will often use a natural notion of a ``sub-phantom'' of $\mf{p}$, which we define next. Let $i,r'\in \{0,\ldots, r\}$ such that $i+r'\leq r$ and let $X_0\subseteq Z_i$.  We denote by $\mf{p}[X_0;i,r']$ the $(X_0,d,r')$-phantom 
$(X_0,\ldots, X_{r'};\Upsilon_i:i\in [r'])$ which is defined recursively, as follows. For each $j\in [r']$ with $X_{j-1}$ already defined, let 
$$X_{j}=X_{j-1}\cup \left(\bigcup_{e\in E(G[X_{j-1}])}\Gamma_{i+j}(e)\right);$$
and let
$$\Upsilon_{j}=\Gamma_{i+j}{\bigg|}_{E(G[X_{j-1}])}.$$

See Figure~\ref{fig:phantom}.

We leave it to the reader to check that the above definition does yield a $(X_0,d,r')$-phantom. Also, for each $j\in \{0,\ldots, r'\}$, we have $X_{j}\subseteq Z_{i+j}$, and for each $j\in [r']$, we have $X_{i}\setminus X_{i-1}\subseteq Z_{i+j}\setminus Z_{i+j-1}$.

As mentioned at the beginning of this section, our first step is to show that every graph $G\in \mathcal{E}_t$ with sufficiently large treewidth contains a large phantom. We need this to be true both when $Z_0$ is a $2$-clique and when $Z_0$ is a $3$-clique. Therefore, we prove:

\begin{theorem}\label{thm:twtokalei} For all integers $d,t
\geq 1$ and $r_1\geq 0$, there exists an integer $\Omega=\Omega(d,r_1,t)\geq 1$ with the following property. Let $G\in \mathcal{E}_t$ be a graph of treewidth more than $\Omega$. Then there is a $2$-clique $Z_0^1$ as well as a $(Z^1_0,d,r_1)$-phantom in $G$. Consequently, for every $r_2\in [r_1-1]$, there is a $3$-clique $Z^2_0$ as well as a $(Z^2_0,d,r_2)$-phantom in $G$.
\end{theorem}

For the second step, we prove the following two results. The first will be used in the proof of Theorem~\ref{thm:maincone2} and the second will be used in the proof of Theorem~\ref{thm:maincrystal2}.

\begin{theorem}\label{thm:doublecone}
     Let $f,g,h,t\geq 1$ be integers. Let $G$ be a $(K_{2,2},K_t)$-free graph and let $Z\subseteq V(G)$ with $|Z|\leq h$. Let $Z_0=\{z_1,z_2,z\}\subseteq Z$ be a $3$-clique such that $N_Z(z)=\{z_1,z_2\}$. Let $\mf{p}=(Z_0,\cdots,Z_f;\Gamma_i:i\in [f])$ be a $(Z_0,f+g+2^ht,f)$-phantom in $G$ such that $Z_{f}\cap Z=Z_0$. Then one of the following holds.
\begin{enumerate}[{\rm(a)}]
        \item\label{thm:doublecone_a} There exists a $(z_1,z_2,1,g)$-crystal $\mf{c}$ in $G[Z_f]$ such that $V(\mf{c})$ is anticomplete to $Z\setminus Z_0$.
        \item \label{thm:doublecone_b} $G$ has a subgraph isomorphic to $\cone(\cone(T_{f,f}))$.
        \end{enumerate}
\end{theorem}

\begin{theorem}\label{thm:crystal}
Let $f,g,t\geq 1$ be integers. Let $G$ be a $K_t$-free graph. Assume that there is a $(Z_0,f+g,t)$-phantom in $G$ for some $2$-clique $Z_0\subseteq V(G)$. Then there is an $(f,g)$-crystal in $G$.
\end{theorem}

Assuming Theorems~\ref{thm:twtokalei}, \ref{thm:doublecone} and \ref{thm:crystal} to hold true,  we now embark on proving Theorems~\ref{thm:maincone2} and \ref{thm:maincrystal2}. First, we restate and prove the second as it is almost immediate:

\setcounter{section}{2}
\setcounter{theorem}{6}

\begin{theorem}
    For all integers $f,g,t\geq 1$, there exists an integer $\tau_2=\tau_2(f,g,t)\geq 1$ such that for every graph $G\in \mathcal{E}_t$ of treewidth more than $\tau$, there is an $(f,g)$-crystal in $G$.
\end{theorem}
\begin{proof}
Let $\tau_2(f,g,t)=\Omega(f+g,t,t)$
where $\Omega(\cdot,\cdot,\cdot,\cdot)$ comes from Theorem~\ref{thm:twtokalei}. Let $G\in \mathcal{E}_t$ be a graph of treewidth more than $\tau_2$. Then, by Theorem~\ref{thm:twtokalei}, 
there is a $2$-clique $Z_0$ as well as a $(Z_0,f+g,t)$-phantom in $G$. Since $G$ is $K_t$-free, it follows from Theorem~\ref{thm:crystal} that there is an $(f,g)$-crystal in $G$. This completes the proof of Theorem~\ref{thm:maincrystal2}.
\end{proof}

\setcounter{section}{3}
\setcounter{theorem}{3}

The proof of Theorem~\ref{thm:maincone2} needs a little more work. Note that, by definition, every $2$-tree $\nabla$ on three or more vertices has a degree-two ``simplicial'' vertex, that is, a vertex $z$ whose neighborhood is a $2$-clique. There is a useful strengthening of this fact: if $\nabla$ has four or more vertices, then we can arrange for $z$ to be a ``side vertex'' from a ``side crystal'' in $\nabla$.
 
Let us make this precise. Given a graph $H$, we say a vertex $z\in V(H)$ is \textit{crystallized} if there is a $2$-clique $\{z_1,z_2\}\subseteq N_H(z)$ with the following specifications. 
\begin{enumerate}[(C1), leftmargin=15mm, rightmargin=7mm]
\item\label{C1} $N_{H}(z)\setminus \{z_1,z_2\}$ is a non-empty stable in $G$.
    \item\label{C2} There exists a partition $(S_1,S_2)$ of $N_{H}(z)\setminus \{z_1,z_2\}$ such that for each $i\in \{1,2\}$ and every vertex $x\in S_i$, we have $N_H(x)=\{z_i,z\}$.
\end{enumerate}
In particular, every vertex in $N_{H}(z)\setminus \{z_1,z_2\}$ is a degree-two simplicial vertex in $H$. We deduce: 
\begin{theorem}\label{thm:crystallizedvertex} Every $2$-tree $\nabla$ with $|V(\nabla)|\geq 4$ has a crystallized vertex.
\end{theorem}
\begin{proof}
    We proceed by induction on $|V(\nabla)|=h\geq 4$. If $h=4$, then $\nabla$ is a diamond and both degree-three vertices of $H$ are crystallized. Assume that $h>4$. Let $\varpi_{\nabla}$ be as in the definition of a $2$-tree,  let $v=\varpi_{\nabla}(1)$ and let $\nabla'=\nabla\setminus \{v\}$. Then $N_{\nabla}(v)\subseteq V(\nabla')$ is $2$-clique and $\nabla'$ is a $2$-tree on $h-1\geq 4$ vertices. By the induction hypothesis, $\nabla'$ has a crystallized $z$ in $\nabla'$. In particular, there exist a $2$-clique $\{z_1,z_2\}\subseteq N_{\nabla'}(z)$ and a partition $(S_1,S_2)$ of $N_{\nabla'}(z)\setminus \{z_1,z_2\}$ satisfying \ref{C1} and \ref{C2}. Now, if $N_{\nabla}(v)\cap (N_{\nabla'}[z]\setminus \{z_1,z_2\})=\emptyset$, then $z$ is a crystallized vertex in $\nabla$ (still with $\{z_1,z_2\}\subseteq N_{\nabla}(z)$ and $(S_1,S_2)$ satisfying \ref{C1} and \ref{C2}). Otherwise, since $N_{\nabla}(v)$ is a $2$-clique in $\nabla'$, there are two possibilities:
    \begin{itemize}
        \item There exists $i\in \{1,2\}$ and $x\in S_i$ such that either $N_{\nabla}(v)=\{x,z_i\}$ or $N_{\nabla}(v)=\{x,z\}$.
        \item There exists $i\in \{1,2\}$ such that $N_{\nabla}(v)=\{z_i,z\}$.
    \end{itemize}
    In the former case, $x$ is a crystallized vertex in $\nabla$ with the $2$-clique $\{z_i,z\}$ contained in $N_{\nabla}(x)$ and the partition $(\{v\},\emptyset)$ of $N_{\nabla}(x)\setminus \{z_i,z\}$ satisfying \ref{C1} and \ref{C2}. In the latter case, let $\{1,2\}\setminus \{i\}=\{i'\}$. Then $z$ is a crystallized vertex in $\nabla$ with the $2$-clique $\{z_1,z_2\}$ contained in $N_{\nabla}(z)$ and the partition $(S_i\cup \{v\},S_{i'})$ of $N_{\nabla}(x)\setminus \{z_1,z_2\}$ satisfying \ref{C1} and \ref{C2}. This completes the proof of Theorem~\ref{thm:crystallizedvertex}.
    \end{proof}

We are now ready to prove Theorem~\ref{thm:maincone2}, which we restate:

\setcounter{section}{2}
\setcounter{theorem}{2}

\begin{theorem}\label{thm:strongmaincone2}
    For all integers $f,t,h\geq 1$, there exists an integer $\tau_1=\tau_1(f,h,t)\geq $ such that for every graph $G\in \mathcal{E}_t$ of treewidth more than $\tau_1$, either $G$ contains every $2$-tree on $h+1$ vertices, or $G$ has a subgraph isomorphic to $\cone(\cone(T_{f,f}))$. 
\end{theorem}
\begin{proof}
  For every $i\in [h]$, let $r_i=(h-i)f$. Let $$\tau_1=\tau_1(f,h,t)=\Omega(f+h+2^ht,r_1,t)\geq 1$$
where $\Omega(\cdot,\cdot,\cdot)$ comes from Theorem~\ref{thm:twtokalei}. Let $G\in \mathcal{E}_t$ be a graph of treewidth more than $\tau_1$. From Theorem~\ref{thm:twtokalei} and the choice of $\tau_1, r_1$ and $r_2$, we deduce that:

\sta{\label{st:2-treeinphantomtobegin}There is a $2$-clique $Z_0^1$ as well as a $(Z^1_0,d,r_1)$-phantom in $G$. Also, if $h\geq 2$, then there is a $3$-clique $Z^2_0$ as well as a $(Z^2_0,d,r_2)$-phantom in $G$.}

Now, assume that $G$ has no subgraph isomorphic to $\cone(\cone(T_{f,f}))$. We wish to show that $G$ contains every $2$-tree on $h+1$ vertices. Instead, we prove the following stronger statement by induction on $i$:

\sta{\label{st:2-treeinphantom} For each $i\in [h]$ and every $2$-tree $\nabla$ on $i+1$ vertices, there is an induced subgraph $Z^i_0$ of $G$ isomorphic to $\nabla$ and a $(Z_0^i,f+h+2^ht,r_i)$-phantom $\mf{p}^i$ in $G$.}

To launch the induction, note that if $i\in \{1,2\}$, then $\nabla$ is an $(i+1)$-vertex complete graph, and so the result follows directly from \eqref{st:2-treeinphantomtobegin}. Thus, we may assume that $3\leq i\leq h$; in particular, $\nabla$ has at least four vertices. By Theorem~\ref{thm:crystallizedvertex}, there is a crystallized vertex $z\in V(\nabla)$, with $\{z_1,z_2\}\subseteq N_{\nabla}(z)$ as well as the partition $(S_1,S_2)$ of $N_{\nabla}(z)\setminus \{z_1,z_2,z\}$ satisfying \ref{C1} and \ref{C2}. Let $\nabla'=\nabla\setminus (S_1\cup S_2)$. It is straightforward to observe that $\nabla'$ is a $2$-tree. Let $|V(\nabla')|=j+1$. Since $S_1\cup S_2\neq \emptyset$, it follows that $|V(\nabla')|<|V(\nabla)|$, and so $2\leq j\leq i-1\leq h-1$. Consequently, we may apply the induction hypothesis, obtaining an induced subgraph $Z^j_0$ of $G$ isomorphic to $\nabla'$ as well as a $(Z_0^j,f+h+2^ht,r_j)$-phantom $\mf{p}^j=(Z^j_0,\ldots, Z^j_{(h-j)f}; \Gamma^j_k:k\in [(h-j)f])$ in $G$. 

Let $Z=Z^j_0$; then we have $|Z|=j+1\leq i\leq h$. Pick an isomorphism $\varphi:V(\nabla')\rightarrow V(Z)$. Let $Z_0=\{\varphi(z_1),\varphi(z_2),\varphi(z)\}\subseteq Z$; then we have $N_z(\varphi(z))=\{\varphi(z_1),\varphi(z_2)\}$. Since $Z_0\subseteq Z=Z^j_0$ and $(h-j)\geq 1$, it follows that $\mf{p}=\mf{p}^j[Z_0;0,f]$ is a $(Z_0,f+h+2^ht,f)$-phantom $(Z_0,\ldots, Z_f; \Gamma_k:k\in [f])$ in $G$. Moreover, since $Z_f\setminus Z_0\subseteq Z^j_f\setminus Z^j_0=Z^j_f\setminus Z$, it follows that $Z_f\cap Z=Z_0$. This allows for an application of Theorem~\ref{thm:doublecone} to $G, Z, Z_0$ and $\mf{p}$. From the assumption that $G$ has no subgraph isomorphic to $\cone(\cone(T_{f,f}))$, we conclude that Theorem~\ref{thm:doublecone}\ref{thm:doublecone_a} holds. Explicitly, there exists a $(\varphi(z_1),\varphi(z_2),1,h)$-crystal $\mf{c}$ in $G[Z_f]$ such that $V(\mf{c})$ is anticomplete to $Z\setminus Z_0$.

Now, since $V(\mf{c})$ is anticomplete to $Z\setminus Z_0$, and since $|S_1|,|S_2|\leq i\leq h$, it follows that there exists an induced subgraph $Z_0^i$ of $G$ which is isomorphic to $\nabla$ such that $(Z\setminus \{\varphi(z)\})\subseteq Z^i_0\subseteq (Z\setminus \{\varphi(z)\})\cup V(\mf{c})$. In particular, we have $Z^i_0\subseteq Z\cup Z_f\subseteq Z^j_f$. Recall also that $j\leq i-1\leq h-1$, and so  we have $f,r_i\in \{0,\ldots, r_j\}$ with $f+r_i\leq r_j$. As a result, $\mf{p}^j=\mf{p}^j[Z^i_0;f,r_i]$ is a $(Z^i_0,f+h+2^ht,r_i)$-phantom in $G$. In conclusion, we have shown that there is an induced subgraph $Z^i_0$ of $G$ isomorphic to $\nabla$ and a $(Z_0^i,f+h+2^ht,r_i)$-phantom $\mf{p}^i$ in $G$. This proves \eqref{st:2-treeinphantom}.
\medskip

Note that Theorem~\ref{thm:maincone2} is identical to \eqref{st:2-treeinphantom} for $i=h$, and so the proof is completed.
\end{proof}
\setcounter{section}{3}
\setcounter{theorem}{4}

It remains to prove Theorems~\ref{thm:twtokalei}, \ref{thm:doublecone} and \ref{thm:crystal}. We prove the latter two in the next section. Then, we prove Theorem~\ref{thm:twtokalei} in Sections~\ref{sec:conjuring}.

\section{Communing with a phantom}\label{sec:phantom}
This section is devoted to the proofs of Theorems~\ref{thm:doublecone} and \ref{thm:crystal}. Instead of each one, we prove a corresponding stronger statement tailored to an inductive argument. Specifically, instead of Theorem~\ref{thm:doublecone}, we prove the following. It is readily observed that Theorem~\ref{thm:doublecone} is a corollary of Theorem~\ref{thm:phantomtodoublecone} below -- we leave this to the reader to check.

\begin{theorem}\label{thm:phantomtodoublecone}
    Let $d,g,h,t\geq 1$ and $r\geq 0$ be integers. Let $G$ be a $(K_{2,2},K_t)$-free graph and let $Z\subseteq V(G)$ with $|Z|\leq h$. Let $Z_0=\{z_1,z_2,z\}\subseteq Z$ be a $3$-clique such that $N_Z(z)=\{z_1,z_2\}$. Let $\mf{p}=(Z_0,\cdots,Z_r;\Gamma_i:i\in [r])$ be a $(Z_0,d+g+2^ht,r)$-phantom in $G$ such that $Z_{r}\cap Z=Z_0$. Then one of the following holds.
\begin{enumerate}[{\rm(a)}]
        \item\label{thm:phantomtodoublecone_a} There exists a $(z_1,z_2,1,g)$-crystal $\mf{c}$ in $G[Z_r]$ such that $V(\mf{c})$ is anticomplete to $Z\setminus Z_0$.
        \item \label{thm:phantomtodoublecone_b} $G$ has a subgraph $U$ isomorphic to $T_{d
        ,r}$ which satisfies the following.
        \begin{itemize}
        \item $U$ contains $z$ as its root.
        \item $\{z_1,z_2\}$ is disjoint from and complete to $V(U)$ in $G$.
            \item Let $i\in [r]$, let $u\in V(U)$ be at distance $i$ in $U$ from $z$ and let $u^-$ be the parent of $u$ in $U$. Then $u\in \Gamma_{i}(u^-z_1)\cup \Gamma_{i}(u^-z_2)\subseteq Z_{i}\setminus Z_{i-1}$. In particular, we have $V(U)\cap Z=\{z\}$.
        \end{itemize}
    \end{enumerate}
\end{theorem}
\begin{proof}
   The proof is by induction on $r$. If $r=0$, then the subgraph $U$ of $G$ with $V(U)=\{z\}$ satisfies \ref{thm:phantomtodoublecone}\ref{thm:phantomtodoublecone_b}. Assume that $r\geq 1$. For each $i\in \{1,2\}$, let $K_i$ be the set of all vertices in $\Gamma_1(z_iz)$ with at least one neighbor in $Z\setminus Z_0$.

    \sta{\label{st:Lismall} We have $|K_1|,|K_2|< 2^ht$.}

    For suppose $|K_i|\geq 2^ht$ for some $i\in \{1,2\}$. Since  $|Z\setminus Z_0|<h$, it follows that there exists a $t$-subset $S$ of $K_i$ and a non-empty subset $Y$ of $Z\setminus Z_0$ such that for every $x\in S$, we have $N_{Z\setminus Z_0}(x)=Y$. Pick a vertex $y\in Y$; then $y,z$ are not adjacent in $G$. Since $G$ is $K_t$-free, it follows that there are two vertices $x,x'\in S$ which are not adjacent in $G$. Also, since $x,x'\in S\subseteq K_i\subseteq \Gamma_1(z_iz)$,  it follows that $z$ is adjacent to both $x$ and $x'$. But now $G[\{x,x',y,z\}]$ is isomorphic to $K_{2,2}$, a contradiction. This proves \eqref{st:Lismall}.
\medskip

By \eqref{st:Lismall}, there exist $L_1\subseteq \Gamma_1(z_2z)$ and $L_2\subseteq \Gamma_1(z_1z)$ with $|L_1|=|L_2|=d+g$ such that $L_1\cup L_2 $ is anticomplete to $Z\setminus Z_0$ in $G$. It follows that either 
\begin{itemize}
    \item for every $i\in \{1,2\}$, there exists $M_i\subseteq L_i$ with $|M_i|=g$ such that $z_i$ is anticomplete to $M_i$ in $G$; or
    \item there exist $j\in \{1,2\}$ and $N_j\subseteq L_j$ with $|N_j|=d$ such that $z_j$ is complete to $N_j$ in $G$.
\end{itemize}
In the former case, $\mf{c}=(M_1,z,M_2)$ is a $(z_1,z_2,1,g)$-crystal in $G[Z_1]$ such that $V(\mf{c})\subseteq L_1\cup L_2\cup \{z\}$ is anticomplete to $Z\setminus Z_0$, and so \ref{thm:phantomtodoublecone}\ref{thm:phantomtodoublecone_a} holds. Therefore, we may assume that the second bullet above holds, and by symmetry, we may also assume that $j=1$. Since $N_1\subseteq \Gamma_1(z_2z)$, it follows that $N_1$ is a $d$-subset of $V(G)\setminus (Z\cup \Gamma_1(z_1z_2))$ and $\{z_1,z_2\}$ is complete to $N_1$.

For every $z'\in N_1$, let 
$Z_{z'}=(Z\setminus \{z\})\cup \{z'\}$ and let $Z_{0,z'}=\{z_1,z_2,z'\}$; then we have $Z_{0,z'}\subseteq Z_1$. Since $r\geq 1$, it follows that $$\mf{p}_{z'}=\mf{p}[Z_{0,z'};1,r-1]=(Z_{0,z'},\cdots,Z_{r-1,z'};\Gamma_{i,z'}:i\in [r-1])$$
is
a $(Z_{0,z'},d+g+2^ht,r-1)$-phantom
in $G$ such that $Z_{i,z'}\subseteq Z_{i+1}$ for all $i\in [r-1]$. Moreover, since $Z_{r-1,z'}\subseteq Z_{r}$ and $Z_{r}\cap Z=Z_0$, it follows that:
$$Z_{r-1,z'}\cap Z_{z'}\subseteq Z_r\cap ((Z\setminus \{z\})\cup \{z'\})=\{z_1,z_2,z'\}=Z_{0,z'};$$
and thus $Z_{r-1,z'}\cap Z_{z'}=Z_{0,z'}$ (see Figure~\ref{fig:doublecone}). Consequently, from the induction hypothesis, we deduce that:

\sta{\label{st:ihfordoublecone} For every vertex $z'\in N_1$, one of the following holds.
 \begin{itemize}
\item There exists a $(z_1,z_2,1,g)$-crystal $\mf{c}$ in $G[Z_{r-1,z'}]$ such that $V(\mf{c})$ is anticomplete to $Z_{z'}\setminus Z_{0,z'}$.
        \item $G$ has a subgraph $U_{z'}$ isomorphic to $T_{d
        ,r-1}$ which satisfies the following.
        \begin{itemize}
        \item $U_{z'}$ contains $z'$ as its root.
             \item $\{z_1,z_2\}$ is disjoint from and complete to $V(U_{z'})$ in $G$.
            \item Let $i\in [r-1]$, let $u\in V(U_{z'})$ be at distance $i$ in $U_{z'}$ from $z'$ and let $u^-$ be the parent of $u$ in $U_{z'}$. Then $u\in \Gamma_{i,z'}(u^-z_1)\cup \Gamma_{i,z'}(u^-z_2)\subseteq Z_{i,z'}\setminus Z_{i-1,z'}$. In particular, we have $V(U_{z'})\cap Z_{z'}=\{z'\}$.
        \end{itemize}
 \end{itemize}
}
\begin{figure}[t!]
\centering
\includegraphics[scale=0.7]{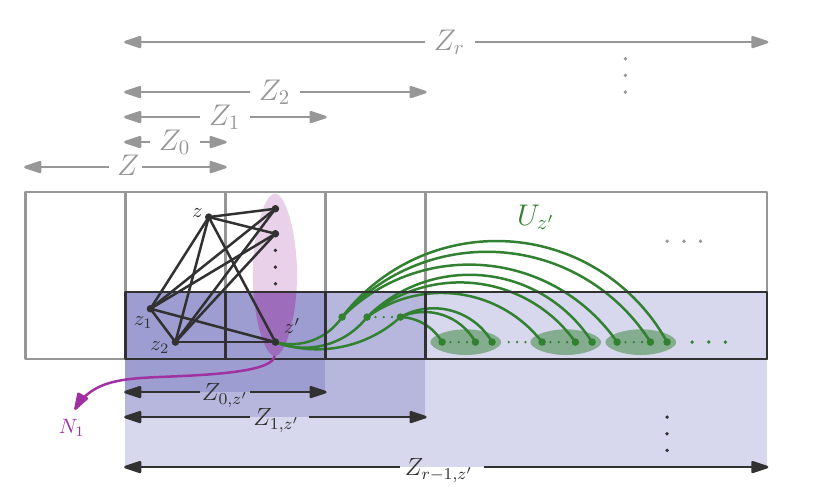}
    \caption{Proof of Theorem~\ref{thm:doublecone}.}
    \label{fig:doublecone}
\end{figure}

Note that if some vertex $z'\in N_1$ satisfies the first bullet of \eqref{st:ihfordoublecone}, then \ref{thm:phantomtodoublecone}\ref{thm:phantomtodoublecone_a} holds because $Z_{r-1,z'}\subseteq Z_r$ and $Z_{z'}\setminus Z_{0,z'}=Z\setminus Z_0$. Thus, we may assume that for every vertex $z'\in N_1$, the second bullet of \eqref{st:ihfordoublecone} holds (again, see Figure~\ref{fig:doublecone}). We also claim that:

\sta{\label{st:vdisjoint} The sets $\{V(U_{z'}):z'\in N_1\}$ are pairwise disjoint.}

Suppose for a contradiction that there are distinct vertices $z',z''\in N_1$ for which $U_{z'}$ and $U_{z''}$ share some vertex $u$. Then $u,z,z''$ are all distinct, and so there are $i',i''\in [r-1]$ such that $u$ is at distance $i'$ in $U_{z'}$ from $z'$, and $u$ is at distance $i''$ in $U_{z''}$ from $z''$. Let $u$ be chosen such that $\min\{i',i''\}$ is as small as possible. Let $u'$ and $u''$ be the parents of $u$ in $U_{z'}$ and $U_{z''}$, respectively. Then, by the choice of $u$, the vertices $u'$ and $u''$ are distinct. From the definition of $\mf{p}_{z'}$ and the second bullet of \eqref{st:ihfordoublecone}, it follows that:
$$u\in \Gamma_{i'+1}(u'z_1)\cup \Gamma_{i'+1}(u'z_2)\subseteq Z_{i'+1}\setminus Z_{i'};$$ 
$$u\in \Gamma_{i''+1}(u''z_1)\cup \Gamma_{i''+1}(u''z_2)\subseteq Z_{i''+1}\setminus Z_{i''}.$$
Particularly, we have $(Z_{i'+1}\setminus Z_{i'})\cap (Z_{i''+1}\setminus Z_{i''})\neq \emptyset$, which yields $i'=i''$. But then $u'z_1,u'z_2,u''z_1,u''z_2$ are pairwise distinct edges of $G[Z_{i'}]$ (because $u',u'',z_1,z_2$ are pairwise distinct vertices), for which we have
$$(\Gamma_{i'+1}(u'z_1)\cup \Gamma_{i'+1}(u'z_2))\cap (\Gamma_{i'+1}(u''z_1)\cup \Gamma_{i'+1}(u''z_2))\neq\emptyset.$$
This violates the second bullet of \ref{P2} from the definition of a phantom, hence completing the proof of \eqref{st:vdisjoint}.
\medskip

Now, let $U$ be the subgraph of $G$ with $V(U)=(\bigcup_{z'\in N_1}V(U_{z'}))\cup \{z\}$ and $E(U)=(\bigcup_{z'\in N_1}E(U_{z'}))\cup \{zz':z'\in N_1\}$. Since $z$ is complete to $N_1$ and $|N_1|=d$, it follows from the second bullet of \eqref{st:ihfordoublecone} (specifically, the first dash) that $U$ is isomorphic to $T_{d,r}$ and $U$ contains $z$ as the root. Also, from the second dash of the second bullet of \eqref{st:ihfordoublecone} and the fact that $\{z_1,z_2\}$ is complete to $z$, it follows that $\{z_1,z_2\}$ is complete to $V(U)$ in $G$. In conclusion, we have shown that $U$ satisfies the first and the second bullets of \ref{thm:phantomtodoublecone}\ref{thm:phantomtodoublecone_b}.

For the third bullet of \ref{thm:phantomtodoublecone}\ref{thm:phantomtodoublecone_b}, let $i\in [r]$, let $u\in V(U)$ be at distance $i$ in $U$ from $z$ and let $u^-$ be the parent of $u$ in $U$. We wish to show that $u\in \Gamma_{i}(u^-z_1)\cup \Gamma_{i}(u^-z_2)$. To see this, note that if $i=1$, then we have $u^-=z$ and $u\in N_1\subseteq L_1\subseteq \Gamma_1(u^-z_2)$, as required. Otherwise, there exists $z'\in N_1$ such that $u,u^-\in V(U_{z'})$, where $u$ is at distance $i-1\in [r-1]$ from $z'$ is $U_{z'}$ and $u^-$ is the parent of $u$ in $U_{z'}$. From the second bullet of \eqref{st:ihfordoublecone} (specifically, the second dash) along with the definition of  $\mf{p}_{z'}$, we deduce that:
$$u\in \Gamma_{i-1,z'}(u^-z_1)\cup \Gamma_{i-1,z'}(u^-z_2)=\Gamma_{i}(u^-z_1)\cup \Gamma_{i}(u^-z_2).$$
Hence, the second bullet of \ref{thm:phantomtodoublecone}\ref{thm:phantomtodoublecone_b} holds for $U$. This completes the proof of Theorem~\ref{thm:phantomtodoublecone}.\end{proof}

Next, we prove the following strengthening of Theorem~\ref{thm:crystal}. The reader can easily check that Theorem~\ref{thm:phantomtocrystal} below does imply Theorem~\ref{thm:crystal}.

\begin{theorem}\label{thm:phantomtocrystal}
Let $f,g\geq 1$ and $r\geq 0$ be integers, let
  $G$ be a graph and let $Z_0$ be a $2$-clique in $G$. Let $\mf{p}=(Z_0,\cdots,Z_r;\Gamma_i:i\in [r])$ be a $(Z_0,f+g,r)$-phantom in $G$. Then one of the following holds.
    \begin{enumerate}[{\rm (a)}]
        \item\label{thm:phantomtocrystal_a} There is an $(f,g)$-crystal in $G[Z_{r}]$.
        \item\label{thm:phantomtocrystal_b} There are $g$ pairwise disjoint $r$-cliques $K_1,\ldots, K_g\subseteq Z_{r}\setminus Z_0$ in $G$ such that
        $Z_0$ is complete to $K_1\cup\cdots \cup K_g$.
    \end{enumerate}
\end{theorem}
\begin{proof}
    The proof is by induction on $r$. Note that \ref{thm:phantomtocrystal}\ref{thm:phantomtocrystal_b} holds trivially for $r=0$. Assume that $r\geq 1$. Let $Z_0=\{z_1,z_2\}$. Let $j\in \{1,2\}$ and $z\in \Gamma_1(z_1z_2)$ be fixed. Define $Z^j_{0,z}=(Z_0\setminus \{z_j\})\cup \{z\}$; then $Z^j_{0,z}\subseteq Z_1$ is a $2$-clique in $G$. Since $r\geq 1$, it follows that $$\mf{p}^j_{z}=\mf{p}[Z^j_{0,z};1,r-1]=(Z^j_{0,z},\cdots,Z^j_{r-1,z};\Gamma_{i,z}:i\in [r-1])$$
is
a $(Z^j_{0,z},f+g,r-1)$-phantom
in $G$ such that $Z^j_{i,z}\subseteq Z_{i+1}$ for all $i\in [r-1]$. We deduce from the induction hypothesis applied to $\mf{p}^j_{z}$ that:

\sta{\label{st:ihforcrystal} For every $j\in \{1,2\}$ and every $z\in \Gamma_1(z_1z_2)$, one of the following holds.
 \begin{itemize}
 \item There is an $(f,g)$-crystal in $G[Z^j_{r-1,z}]$.
        \item 
        There are $g$ pairwise disjoint $(r-1)$-cliques $K^j_{1,z},\ldots, K^j_{g,z}\subseteq Z^j_{r-1,z}\setminus Z^j_{0,z}\subseteq Z_{r}\setminus Z_0$ in $G$ such that
        $Z^j_{0,z}$ is complete to $K^j_{1,z}\cup\cdots \cup K^j_{g,z}$.
 \end{itemize}
}
Now, if there exist $j\in \{1,2\}$ and $z\in \Gamma_1(z_1z_2)$ which satisfy the first bullet of \eqref{st:ihforcrystal}, then \ref{thm:phantomtocrystal}\ref{thm:phantomtocrystal_a} holds because $Z^j_{r-1,z}\subseteq Z_r$. Consequently, we may assume that for all $j\in \{1,2\}$ and $z\in \Gamma_1(z_1z_2)$, the second bullet of \eqref{st:ihforcrystal} holds. Henceforth, for every $j\in \{1,2\}$ and every $z\in \Gamma_1(z_1z_2)$, let  $K^j_{1,z},\ldots, K^j_{g,z}$
be the $g$ pairwise disjoint $(r-1)$-cliques given by the second bullet of \eqref{st:ihforcrystal} (see Figure~\ref{fig:crystalphantom}).
\begin{figure}
    \centering
    \includegraphics[scale=0.7]{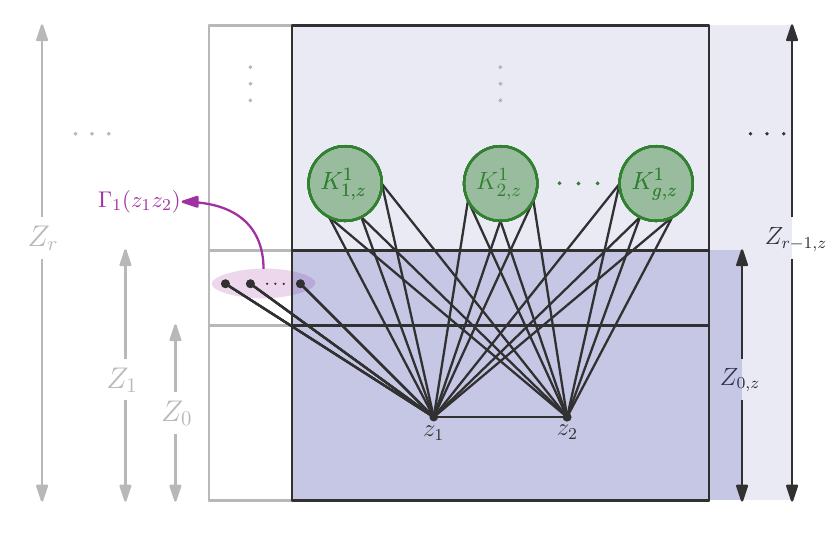}
    \caption{Proof of Theorem~\ref{thm:phantomtocrystal}.}
    \label{fig:crystalphantom}
\end{figure}
\medskip

Moreover, the following is immediate from the definition $\mf{p}^j_{z}$ and the fact that $\mf{p}$ satisfies (the second bullet of) \ref{P2} in the definition of a phantom:

\sta{\label{st:vdisjoint2} The $2(f+g)$ sets $\{Z^j_{r-1,z}\setminus Z^j_{0,z}:j\in \{1,2\}, z\in \Gamma_1(z_1z_2)\}$ are pairwise disjoint.}

Now, since $|\Gamma_1(z_1z_2)|=f+g$, it follows that either
\begin{itemize}
    \item there is an $f$-subset $X$ of $\Gamma_1(z_1z_2)$ such that for every $x\in X$, every $j\in \{1,2\}$ and every $k\in [g]$, the vertex $z_j$ has a non-neighbor $u^j_{k,x}\in K^j_{k,x}$; or
    \item there is a $g$-subset $\{y_1,\ldots, y_g\}$ of $\Gamma_1(z_1z_2)$ such that for every $i\in [g]$, there exist $j_i\in \{1,2\}$ and $k_i\in [g]$ for which $z_{j_i}$ is complete to $K^{j_i}_{k_i,y_i}$.
\end{itemize}
In the former case, for each $x\in X$, let $$U_{1,x}=\{u^1_{2,x}\ldots, u^2_{g,x}\};$$
$$U_{2,x}=\{u^1_{1,x}\ldots, u^1_{g,x}\}.$$
Then, the second bullet of \eqref{st:ihforcrystal} combined with \eqref{st:vdisjoint2} implies that $(U_{1,x},x,U_{2,x}:x\in X)$ is a $(z_1,z_2,f,g)$-crystal in $G[Z_r]$, and so \ref{thm:phantomtocrystal}\ref{thm:phantomtocrystal_a} holds. 

In the latter case, for each $i\in [g]$, let
$$K_i=K^{j_i}_{k_i,y_i}\cup \{y_i\}.$$
Then, by the second bullet of \eqref{st:ihforcrystal} combined with \eqref{st:vdisjoint2}, $K_1,\ldots, K_g\subseteq Z_r\setminus Z_0$ are $g$ pairwise disjoint $r$-cliques in $G$ such that $Z_0$ is complete to $K_1\cup \cdots \cup K_g$. Hence, \ref{thm:phantomtocrystal}\ref{thm:phantomtocrystal_b} holds. This completes the proof of Theorem~\ref{thm:phantomtocrystal}.
\end{proof}

\section{Conjuring a phantom}\label{sec:conjuring}
In this section, we prove:

\setcounter{section}{3}
\setcounter{theorem}{0}
\begin{theorem}For all integers $d,t
\geq 1$ and $r_1\geq 0$, there exists an integer $\Omega=\Omega(d,r_1,t)\geq 1$ with the following property. Let $G\in \mathcal{E}_t$ be a graph of treewidth more than $\Omega$. Then there is a $2$-clique $Z_0^1$ as well as a $(Z^1_0,d,r_1)$-phantom in $G$. Consequently, for every $r_2\in [r_1-1]$, there is a $3$-clique $Z^2_0$ as well as a $(Z^2_0,d,r_2)$-phantom in $G$.
\end{theorem}
\setcounter{section}{5}
\setcounter{theorem}{0}

This relies on a number of definitions and results, listed below, from our earlier papers with several coauthors \cite{twviii, twvii, twxi}.

\begin{theorem}[Abrishami, Alecu, Chudnovsky, Hajebi and Spirkl, see Corollary 2.4 in \cite{twviii} and also Theorem 6.5 in \cite{twvii}]\label{thm:noblocksmalltw_Ct}
For all integers $k,t\geq 1$, there exists an integer $\beta=\beta(k,t)$ such that every (theta, prism, $K_t$)-free graph of treewidth more than $\beta$ has two non-adjacent vertices $a,b$ as well as a collection $\mathcal{P}$ of $k$ pairwise internally disjoint paths in $G$ from $a$ to $b$. \end{theorem}

\begin{theorem}[Abrishami, Alecu, Chudnovsky, Hajebi and Spirkl \cite{twviii}]\label{banana}
    For all integers $t,\nu\geq 1$, there exists an integer $\psi=\psi(t,\nu)\geq 1$ with the following property. Let $G$ be a (theta, prism, $K_t$)-free graph, let $a,b\in V(G)$ be distinct and non-adjacent and let $\mathcal{P}$ be a collection of pairwise internally disjoint paths in $G$ from $a$ to $b$ with $|\mathcal{P}|\geq \psi$. For each $P\in \mathcal{P}$, let $x_{P}$ be the neighbor of $a$ in $P$ (so $x_P\neq b$). Then there exist $P_1,\ldots, P_{\nu}\in \mathcal{P}$ such that:
    \begin{itemize}
    \item $\{x_{P_1},\ldots, x_{P_{\nu}},b\}$ is a stable set in $G$; and
    \item for all $i,j\in [\nu]$ with $i<j$, $x_{P_i}$ has a neighbor in $P_j^*\setminus \{x_{P_j}\}$.
    \end{itemize}
\end{theorem}

\begin{theorem}[Alecu, Chudnovsky, Hajebi and Spirkl, see Theorem 3.2 in \cite{twxi}]\label{thm:evenwheeltheta}
 Let $G$ be a (theta, even-wheel)-free graph, let $C$ be a hole in $G$ and let $z_1,z_2\in V(G)\setminus C$ be distinct and adjacent, each with at least three neighbors in $C$. Then $z_1$ and $z_2$ have a common neighbor in $C$.
\end{theorem}

For a graph $G$ and an edge $z_1z_2\in E(G)$, we define the \textit{$z_1z_2$-contraption} of $G$ to be the minor of $G$ (without parallel edges) obtained by first contracting the edge $z_1z_2$ into a new vertex $z$, and then removing from the resulting graph the edges between $z$ and the symmetric difference of $N_G(z_1)$ and $N_G(z_2)$. The following was the most critical tool in the proofs of Theorems~\ref{thm:tw11k4} and \ref{thm:tw11conedforest}  \cite{twxi}:

\begin{theorem}[Alecu, Chudnovsky, Hajebi and Spirkl, see Theorem 3.1 in \cite{twxi}]\label{thm:<3}
    Let $G\in \mathcal{E}$ be a graph and let $z_1z_2\in E(G)$ such that $N_G(z_1)\cap N_G(z_2)$ is a stable set of vertices of degree at most three in $G$. Then the $z_1z_2$-contraption of $G$ belongs to $\mathcal{E}$.
\end{theorem}

We need the following two definitions from \cite{twxi}. Given a graph $G$ and an integer $w\geq 1$, a \textit{$w$-kaleidoscope} in $G$ is a 4-tuple $(a,x,y,\mathcal{W})$ where:
\begin{enumerate}[(K1), leftmargin=15mm, rightmargin=7mm]
\item\label{K1} we have $a,x,y\in V(G)$, and $x\dd a\dd y$ is a path in $G$ (so $x$ and $y$ are distinct and non-adjacent);
\item\label{K2} $\mathcal{W}$ is a set of $w$ pairwise internally disjoint paths in $G\setminus \{a\}$ from $x$ to $y$; and
\item\label{K3} for every $W\in \mathcal{W}$, $a$ is anticomplete to $W^*$ in $G$.
\end{enumerate}

Furthermore, given a subset $Z\subseteq V(G)$ and an integer $d\geq 1$, we say that $Z$ \textit{is $d$-mirrored by $(a,x,y,\mathcal{W})$} if:
\begin{enumerate}[(M1), leftmargin=15mm, rightmargin=7mm]
\item\label{M1} $Z$ is disjoint from $(\bigcup_{W\in \mathcal{W}}V(W))\cup \{a\}$;
\item\label{M2} the vertex $a$ has at most one neighbor in $Z$; and
\item\label{M3} for every $z\in Z$ and every $W\in \mathcal{W}$, $z$ is anticomplete to $N_{W}[x]\cup N_{W}[y]$, and $z$ has at least $d$ distinct neighbors in $W$. In particular, $z$ is anticomplete to $\{x,y\}$.
\end{enumerate}
 We also say a vertex $z\in V(G)$ is \textit{$d$-mirrored by $(a,x,y,\mathcal{W})$} if $\{z\}$ is $d$-mirrored by $(a,x,y,\mathcal{W})$; see Figure~\ref{fig:kalei}. 
\begin{figure}
    \centering
    \includegraphics[scale=0.7]{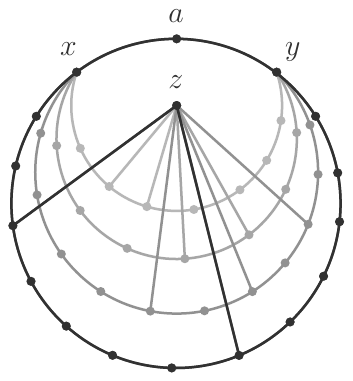}
    \caption{A $4$-kaleidoscope which $2$-mirrors the vertex $z$.}
    \label{fig:kalei}
\end{figure}

The following two results on kaleidoscopes are proved in \cite{twxi}:
\begin{theorem}[Alecu, Chudnovsky, Hajebi and Spirkl, see Theorem 5.1 in \cite{twxi}]\label{thm:<3}\label{thm:kaleidoscopeexists}
    For all integers $d,t,w\geq 1$, there exists $\zeta=\zeta(d,t,w)\geq 1$ with the following property. Let $G$ be a (theta, prism, $K_t$)-free graph, let $a,b\in V(G)$ be distinct and non-adjacent and let $\mathcal{P}$ be a collection of pairwise internally disjoint paths in $G$ from $a$ to $b$ with $|\mathcal{P}|\geq \zeta$.
    Then there exists a $w$-kaleidoscope $(a,x,y,\mathcal{W})$ in $G$ as well as a $2$-clique $Z_0$ in $G$ such that $Z_0$ is $d$-mirrored by $(a,x,y,\mathcal{W})$.
    \end{theorem}

\begin{theorem}[Alecu, Chudnovsky, Hajebi and Spirkl, see Theorem 4.1 in \cite{twxi}]\label{thm:difftodiff}
For all integers $d,t,w\geq 1$, there exists an integer $\kappa=\kappa(d,t,w)\geq 1$ with the following property. Let $G\in \mathcal{E}_t$ be a graph and let $z\in V(G)$ be $1$-mirrored by a $\kappa$-kaleidoscope $(a,x,y,\mathcal{W})$ in $G$. Then there exists $\mathcal{W}'\subseteq \mathcal{W}$ with $|\mathcal{W}'|=w$ such that $z$ is $d$-mirrored by the $w$-kaleidoscope $(a,x,y,\mathcal{W}')$.
\end{theorem}

We also use the following, which is an immediate consequence of Theorem~\ref{classicalramsey}:

\begin{theorem}\label{thm:ramseytournament}
For all integers $c,s\geq 1$, every digraph on at least $c^{c^s}$ vertices contains either an acyclic tournament on $c$ vertices or a stable set of cardinality $s$. 
\end{theorem}

 The next lemma overcomes the bulk of the difficulty in the proof of Theorem~\ref{thm:twtokalei}:

\begin{lemma}\label{lem:gettingphantom}
         For all integers $d,h,t,w\geq 1$, there exists an integer $\xi=\xi(d,h,t,w)\geq 1$ with the following property. Let $G\in \mathcal{E}_t$ be a graph and let $Z_0\subseteq V(G)$ with $|Z_0|\leq h$. Assume that $Z_0$ is $3$-mirrored by a $\xi$-kaleidoscope $(a,x,y,\mathcal{W})$ in $G$. Then there exists a $(Z_0,d,1)$-phantom $(Z_0, Z_1,\Gamma_1)$ in $G$ with $|Z_1|\leq 2dh^2$ and a $w$-subset $\mathcal{W}'$ of $\mathcal{W}$, such that $Z_1$ is $3$-mirrored by the $w$-kaleidoscope $(a,x,y,\mathcal{W}')$.
         \end{lemma}
         \begin{proof}
For fixed $d,t,w\geq 1$, we define the sequence $\{\xi_j:j=0,\ldots, h^2\}$ of positive integers recursively, as follows. Let $n(\cdot,\cdot,\cdot)$, $\psi(\cdot,\cdot)$ and $\kappa(\cdot,\cdot,\cdot)$ be as in Theorems~\ref{ramsey2}, \ref{banana} and \ref{thm:difftodiff}, respectively. Let $\xi_0=\xi_0(d,t,w)=w$. For each $j\in [h^2]$, assuming $\xi_{j-1}$ is defined, let 
$$f_j=f_j(d,t,w)=\kappa(3,t,\xi_{j-1})^{\xi_{j-1}}d+\kappa(3,t,\xi_{j-1});$$
$$g_j=g_j(d,t,w)=f_j^{f_j^{\psi(t,2)}};$$
$$\xi_{j}=\xi_j(d,t,w)=n\left(g_j,3,2,t\right).$$
In particular, observe that the sequence $\{\xi_j:j=0,\ldots, h^2\}$  is increasing.

Define
$$\xi=\xi(d,h,t,w)=\xi_{h^2}(d,t,w)=\xi_{h^2}\geq 1.$$
We prove that the above value of $\xi$ satisfies Lemma~\ref{lem:gettingphantom}. Indeed, we prove a stronger statement tailored to an inductive proof. Let $|E(G[Z_0])|=m$; then we have $0\leq m\leq h^2$. Pick an enumeration $e_1,\ldots, e_m$ of the edges of $G[Z_0]$. Our goal is to show that:

\sta{\label{st:strongerinductphantom1} For each $j\in \{0,1,\ldots, m\}$, there exist $j$ pairwise disjoint $d$-subsets $\{\Gamma_k:k\in [j]\}$ of $V(G)\setminus Z_0$ as well as a $\xi_{m-j}$-subset $\mathcal{W}_j$ of $\mathcal{W}$, such that the following hold.
\begin{itemize}
    \item  For every $k\in [j]$, the ends of $e_k$ are complete to $\Gamma_k$.
    \item The set $Z_0\cup (\cup_{k\in [j]}\Gamma_k)$ is $3$-mirrored by  the $\xi_{m-j}$-kaleidoscope $(a,x,y,\mathcal{W}_j)$.
\end{itemize}}

Let us first prove that \eqref{st:strongerinductphantom1} implies Lemma~\ref{lem:gettingphantom}. To see this, let $\{\Gamma_k:k\in [m]\}$ and $\mathcal{W}'=\mathcal{W}_m$
be as in \eqref{st:strongerinductphantom1} for $j=m$; thus, we have $|Z_1|=|Z_0|+dm\leq 2dh^2$ and $|\mathcal{W}'|=\xi_0=w$. Let $Z_1=Z_0\cup (\cup_{k\in [m]}\Gamma_k)$. Define the map $\Gamma_1$ with domain $E(G[Z_0])$ such that for each $k\in [m]$, we have $\Gamma_1(e_k)=\Gamma_k\subseteq Z_1\setminus Z_0$. Then from \eqref{st:strongerinductphantom1}, and specifically the first bullet of \eqref{st:strongerinductphantom1}, it follows that $(Z_0,Z_1,\Gamma_1)$ is a $(Z_0,d,1)$-phantom in $G$ with $|Z_1|=|Z_0|+dm\leq 2dh^2$, and from the second bullet of \eqref{st:strongerinductphantom1}, it follows that $Z_1$ is $3$-mirrored by the $w$-kaleidoscope $(a,x,y,\mathcal{W}')$, as desired.

We now turn to the proof of \eqref{st:strongerinductphantom1}, which is by induction on $j$. For $j=0$, the result follows from the fact that $\xi_{h^2}\geq \xi_m$ as well as the assumption that $Z_0$ is $3$-mirrored by the $\xi_{h^2}$-kaleidoscope $(a,x,y,\mathcal{W})$. So we may assume that $j\geq 1$. By the induction hypothesis, there exist $j-1$ pairwise disjoint $d$-subsets $\{\Gamma_k:k\in [j-1]\}$ of $V(G)\setminus Z_0$ as well as a $\xi_{m-j+1}$-subset $\mathcal{W}_{j-1}$ of $\mathcal{W}$, such that:
\begin{itemize}
    \item  for every $k\in [j-1]$, the ends of $e_k$ are complete to $\Gamma_k$; and
    \item the set $Z_0\cup (\cup_{k\in [j-1]}\Gamma_k)$ is $3$-mirrored by  the $\xi_{m-j+1}$-kaleidoscope $(a,x,y,\mathcal{W}_{j-1})$.
\end{itemize}
Let $z_1,z_2\in Z_0$ be the ends of $e_j$ and let $Z^+=Z_0\cup (\cup_{k\in [j-1]}\Gamma_k)$. Let $W\in \mathcal{W}_{j-1}$ be fixed. Let $x_W,y_W$ be the neighbors of $x$ and $y$ in $W$, respectively. Consider the hole $C_W=a\dd x\dd W\dd y\dd a$ in $G$. Since $Z^+$ is $3$-mirrored by $(a,x,y,\mathcal{W}_{j-1})$, it follows from \ref{M3} that $z_1$ and $z_2$ each have at least three neighbors in $W\subseteq C_W$, and from \ref{M2} that $a$ is adjacent to at most one of $z_1$ and $z_2$. Thus, by Theorem~\ref{thm:evenwheeltheta}, the vertices $z_1$ and $z_2$ have a common neighbor in $W$. Traversing $W$ from $x$ to $y$, let $z_W$ be the first common neighbor of $z_1$ and $z_2$ in $W$. By \ref{M3}, $\{z_1,z_2\}$ is anticomplete to $\{x,x_W,y,y_W\}$, and in particular we have $z_W\in W\setminus \{x,x_W,y,y_W\}$. (see Figure~\ref{fig:lhbij}).
\begin{figure}[t!]
    \centering
\includegraphics[scale=0.6]{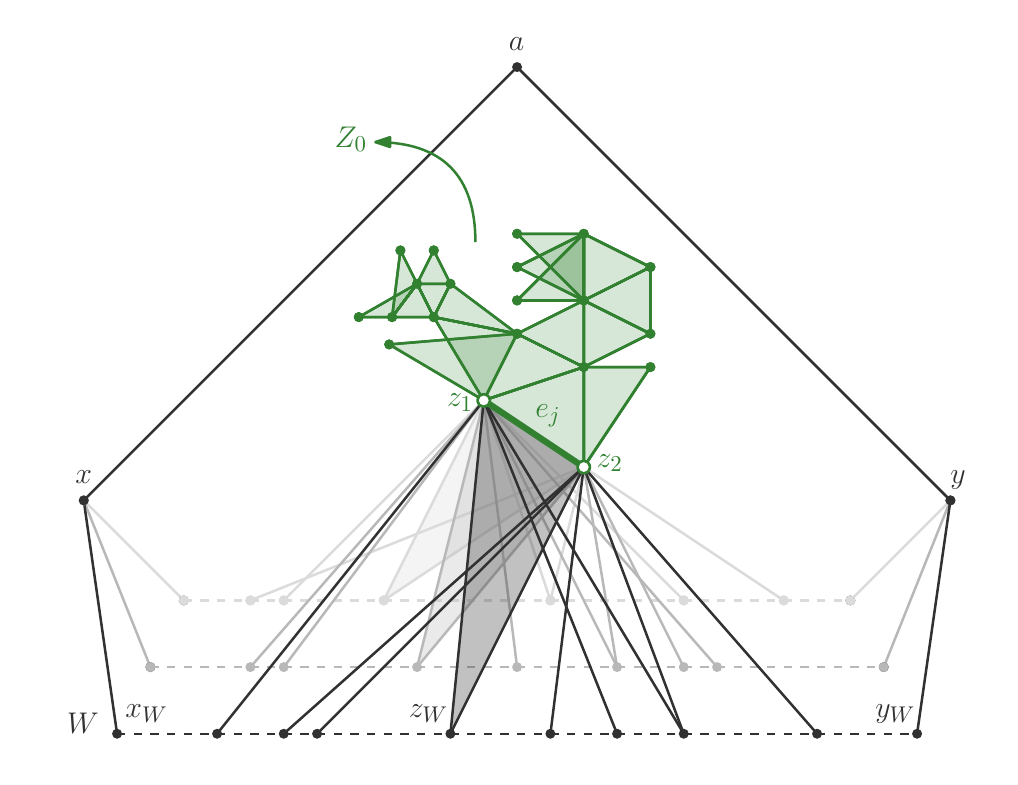}
    \caption{Proof of Theorem~\ref{lem:gettingphantom} (dashed lines represents paths of undetermined length).}
    \label{fig:lhbij}
\end{figure}
\medskip

Since $G$ is $(K_{2,2},K_t)$-free and since $|\mathcal{W}_{j-1}|=\xi_{m-j+1}=n(g_{m-j+1},2,3,t)$, it follows from Lemma~\ref{ramsey2} applied to the sets $\{\{x_W,z_W,y_W\}: W\in \mathcal{W}_{j-1}\}$ that there exists $\mathcal{W}'_{j-1}\subseteq \mathcal{W}_{j-1}$ with $|\mathcal{W}'_{j-1}|=g_{m-j+1}$ such that the sets $\{x_W,z_W,y_W\}$ for $W\in \mathcal{W}'_{j-1}$ are pairwise anticomplete in $G$.

Next, we define a digraph $D$ with vertex set $\mathcal{W}'_{j-1}$ such that for distinct paths $W_1,W_2\in \mathcal{W}'_{j-1}$, the arc $(W_1,W_2)$ is present in $D$ if and only if $z_{W_1}$ has a neighbor in $W_2^*$. We claim that:

\sta{\label{st:transitive} $D$ contains an acyclic tournament on $f_{m-j+1}$ vertices.}

Suppose not. Recall that $V(D)=\mathcal{W}'_{j-1}$ and $|\mathcal{W}'_{j-1}|=g_{m-j+1}$. So we have $|V(D)|=|\mathcal{W}'_{j-1}|=g_{m-j+1}$.
Since $D$ is assumed not to contain an acyclic tournament on $f_{m-j+1}$ vertices, from Theorem~\ref{thm:ramseytournament} and the definition of $g_{m-j+1}$, it follows that there is a stable set in $D$ of cardinality $\psi(t,2)$. As a result, there exists $\mathcal{S}\subseteq \mathcal{W}'_{j-1}$ with $|\mathcal{S}|=\psi(t,2)$ such that for all distinct $W_1,W_2\in \mathcal{S}$, the vertex $z_{W_1}$ is anticomplete to $W_2^*$. Let $J=G[(\bigcup_{W\in \mathcal{S}}V(z_W\dd W\dd x))\cup \{z_1,z_2\}]$ and let $J'$ be the $z_1z_2$-contraption of $J$. Note that we have $J\in \mathcal{E}_t$, and the vertices $z_1,z_2\in V(J)$ are distinct and adjacent with $N_{J}(z_1)\cap N_{J}(z_2)=\{z_W:W\in \mathcal{S}\}$ being a stable set of vertices of degree three in $J$. Consequently, by Theorem~\ref{thm:<3},  we have $J'\in \mathcal{E}$. In fact, $J'$ is $K_t$-free (because $J$ is), and so we have $J'\in \mathcal{E}_t$. Let $z\in V(J')$ be as in the definition of the $z_1z_2$-contraption. Then $N_{J'}(z)=N_{J}(z_1)\cap N_{J}(z_2)=\{z_W:W\in \mathcal{S}\}$ is a stable set of vertices of degree two in $J$. Moreover, $\mathcal{P}=\{z\dd z_W\dd W\dd x:W\in \mathcal{S}\}$ is a collection of $\psi(t,2)$ pairwise internally disjoint paths in $J'\in \mathcal{E}_t$ between the non-adjacent vertices $z$ and $x$.
Applying Theorem~\ref{banana} to $\mathcal{P}$, we deduce that there exist $W_1,W_2\in \mathcal{S}$ for which $z_{W_1}$ has a neighbor in the interior of $z_{W_2}\dd W_2\dd x$ in $J'$. But then $z_{W_1}\in N_{J'}(z)$ has degree at least three in $J'$, a contradiction. This proves \eqref{st:transitive}.

\medskip

By \eqref{st:transitive} and the definition of $f_{m-j+1}$, there are two disjoint subsets $\mathcal{U}$ and $\mathcal{V}$ of $V(D)=\mathcal{W}'_{j-1}$ with $|\mathcal{U}|=d(\kappa(3,t,\xi_{j-1}))^{\xi_{j-1}}$ and $|\mathcal{V}|=\kappa(3,t,\xi_{m-j})$, such that for all $W\in \mathcal{U}$ and $W'\in \mathcal{V}$, the arc $(W,W')$ belongs to $D$. It follows that for every $W\in \mathcal{U}$, the vertex $z_{W}$ is $1$-mirrored by the $\kappa(3,t,\xi_{m-j})$-kaleidoscope $(a,x,y,\mathcal{V}$. Therefore, by Theorem~\ref{thm:difftodiff}, for each $W\in \mathcal{U}$, there exists a $\xi_{m-j}$-subset $\mathcal{V}_W$ of $\mathcal{V}$ such that the vertex $z_{W}$ is $3$-mirrored by the $\xi_{m-j}$-kaleidoscope $(a,x,y,\mathcal{V}_W)$. From this, combined with the fact that $|\mathcal{U}|=d(\kappa(3,t,\xi_{j-1}))^{\xi_{j-1}}$, we deduce the following: 

\sta{\label{st:samenbrs} There exists a $d$-subset $\mathcal{D}$ of $\mathcal{U}$ and a $\xi_{m-j}$-subset of $\mathcal{W}_{j}$ of $\mathcal{V}$ such that for every $W\in \mathcal{D}$, we have $\mathcal{V}_{W}=\mathcal{W}_{j}$.}

Let $\Gamma_j=\{z_{W}:W\in \mathcal{D}\}$. Then $\Gamma_j$ is a $d$-subset of $V(G)\setminus Z^+$ which is complete to $\{z_1,z_2\}$. In particular, $\{\Gamma_k:k\in [j]\}$ are $j$ pairwise disjoint $d$-subsets of $V(G)\setminus Z_0$, and for each $k\in [j]$, the ends of $e_k$ are complete to $\Gamma_k$. By \eqref{st:samenbrs}, $\Gamma_j$ is $3$-mirrored by the $\xi_{m-j}$-kaleidoscope $(a,x,y,\mathcal{W}_j)$. Moreover, $Z^+$ is $3$-mirrored by the $\xi_{m-j}$-kaleidoscope $(a,x,y,\mathcal{W}_j)$, because $\mathcal{W}_{j}\subseteq \mathcal{V}\subseteq \mathcal{W}'_{j-1}\subseteq \mathcal{W}_{j-1}$. In conclusion, we have shown that $Z^+\cup \Gamma_j=Z_0\cup (\cup_{k\in [j]}\Gamma_k)$ is $3$-mirrored by the $\xi_{m-j}$-kaleidoscope $(a,x,y,\mathcal{W}_j)$. Therefore, the sets $\{\Gamma_k:k\in [j]\}$ satisfy the two bullet conditions of \eqref{st:strongerinductphantom1}. This completes the inductive proof of \eqref{st:strongerinductphantom1}, hence finishing the proof of Lemma~\ref{lem:gettingphantom}
\end{proof}

From Lemma~\ref{lem:gettingphantom}, we deduce that:

\begin{theorem}\label{thm:gettingphantom}
         For all integers $d,h,t\geq 1$ and $r\geq 0$, there exist $(r+1)$ positive integers $\{\Xi_i=\Xi_i(d,h,t):i=0,\ldots, r\}$ with $\Xi_0=1$ for which the following holds. Let $G\in \mathcal{E}_t$ be a graph and let $Z_0\subseteq V(G)$ with $|Z_0|\leq h$. Assume that $Z_0$ is $3$-mirrored by a $\Xi_r$-kaleidoscope $(a,x,y,\mathcal{W}_0)$ in $G$. Then,  for each $i\in \{0,\ldots, r\}$, there exists a $(Z_0,d,i)$-phantom $(Z_0,\ldots, Z_i,\Gamma_j:j\in [i])$ in $G$ with $|Z_i|\leq (2d)^{2^i-1}h^{2^i}$ and a $\Xi_{r-i}$-subset $\mathcal{W}_i$ of $\mathcal{W}$, such that $Z_i$ is $3$-mirrored by the $\Xi_{r-i}$-kaleidoscope $(a,x,y,\mathcal{W}_i)$. In particular, there exists a $(Z_0,d,r)$-phantom $(Z_0,\ldots, Z_r,\Gamma_j:j\in [r])$ in $G$.
         \end{theorem}

\begin{proof}

Let $\xi(\cdot,\cdot,\cdot,\cdot)$ be as in Lemma~\ref{lem:gettingphantom}. We define $\{\Xi_i=\Xi_i(d,h,t):i=0,\ldots, r\}$ recursively, as follows. We already know that $\Xi_0=\Xi_0(d,h,t)=1$. For each $i\in [r]$, assuming $\Xi_{i-1}$ is defined, let
$$\Xi_i=\Xi_i(d,h,t)=\xi(d,(2d)^{2^{r-i}-1}h^{2^{r-i}}, t, \Xi_{i-1}).$$
We prove, by induction on $i$, that there exists a $(Z_0,d,i)$-phantom in $G$ which satisfies Theorem~\ref{thm:gettingphantom}.

The case $i=0$ is immediate from the assumptions that $|Z_0|\leq h$ and $Z_0$ is $3$-mirrored by the $\Xi_{r}$-kaleidoscope $(a,x,y,\mathcal{W}_0)$. Assume that $i\geq 1$. By the induction hypothesis, there exists a $(Z_0,d,i-1)$-phantom $(Z_0,\ldots, Z_{i-1},\Gamma_j:j\in [i-1])$ in $G$ with $|Z_{i-1}|\leq (2d)^{2^{i-1}-1}h^{2^{i-1}}$ and a $\Xi_{r-i+1}$-subset $\mathcal{W}_{i-1}$ of $\mathcal{W}$, such that $Z_{i-1}$ is $3$-mirrored by the $\Xi_{r-i+1}$-kaleidoscope $(a,x,y,\mathcal{W}_{i-1})$. 

From $\Xi_{r-i+1}=\xi(d,(2d)^{2^{i-1}-1}h^{2^{i-1}}, t, \Xi_{r-i})$, we can apply Lemma~\ref{lem:gettingphantom} to $Z_{i-1}$ and $(a,x,y,\mathcal{W}_{i-1})$, and deduce that there exists a $(Z_{i-1},d,1)$-phantom $(Z_{i-1}, Z_i,\Gamma_j)$ in $G$ with $|Z_i|\leq 2d((2d)^{2^{i-1}-1}h^{2^{i-1}})^2=(2d)^{2^{i}-1}h^{2^{i}}$ and a $\Xi_{r-i}$-subset $\mathcal{W}_{i}$ of $\mathcal{W}_{i-1}\subseteq \mathcal{W}$, such that $Z_{i}$ is $3$-mirrored by the $\Xi_{r-i}$-kaleidoscope $(a,x,y,\mathcal{W}_i)$. In particular, since $(Z_0,\ldots, Z_{i-1},\Gamma_j:j\in [i-1])$ is a $(Z_0,d,i-1)$-phantom in $G$ and $(Z_{i-1}, Z_i,\Gamma_i)$ is a $(Z_{i-1},d,1)$-phantom in $G$, it follows that $(Z_0,\ldots, Z_{i},\Gamma_j:j\in [i])$ is a $(Z_0,d,i)$-phantom in $G$. The induction is completed,  and so is the proof of Theorem~\ref{thm:gettingphantom}.
\end{proof}

Finally, we give a proof of Theorem~\ref{thm:twtokalei}, which we restate:

\setcounter{section}{3}
\setcounter{theorem}{0}
\begin{theorem}For all integers $d,t
\geq 1$ and $r_1\geq 0$, there exists an integer $\Omega=\Omega(d,r_1,t)\geq 1$ with the following property. Let $G\in \mathcal{E}_t$ be a graph of treewidth more than $\Omega$. Then there is a $2$-clique $Z_0^1$ as well as a $(Z^1_0,d,r_1)$-phantom in $G$. Consequently, for every $r_2\in [r_1-1]$, there is a $3$-clique $Z^2_0$ as well as a $(Z^2_0,d,r_2)$-phantom in $G$.
\end{theorem}
\setcounter{section}{5}
\setcounter{theorem}{9}
\begin{proof}
    Let $\{\Xi_{i}(d,3,t):i=0,\ldots, r_1\}$ be as in Theorem~\ref{thm:gettingphantom}, and let $\Xi=\Xi_{r_1}(d,3,t)$. Let $$\Omega(d,r_1,t)=\beta(\zeta(3,t,\Xi),t);$$
    where $\beta(\cdot,\cdot)$ and $\zeta(\cdot,\cdot, \cdot)$ come from Theorems~\ref{thm:noblocksmalltw_Ct} and \ref{thm:kaleidoscopeexists}. Let $G\in \mathcal{E}_t$ be a graph of treewidth more than $\Omega$.  From the choice of $\Omega$ combined with  Theorems~\ref{thm:noblocksmalltw_Ct} and \ref{thm:kaleidoscopeexists}, it follows that there exists a $\Xi$-kaleidoscope $(a,x,y,\mathcal{W})$ in $G$ as well as a $2$-clique $Z^1_0$ in $G$ such that $Z^1_0$ is $3$-mirrored by $(a,x,y,\mathcal{W})$. Thus, by Theorem~\ref{thm:gettingphantom} and the choice of $\Xi$, there exists a $(Z_0^1,d,r_1)$-phantom $\mf{p}_1=(Z^1_0,\ldots, Z^1_{r_1};\Gamma^1_i: i\in [r_1])$ in $G$, and so $Z^1_0$ satisfies \ref{thm:twtokalei}. 
    
    Now, let $r_2\in [r_1-1]$. Assuming $Z_0^1=\{z_1,z_2\}$, since $d,r_1\geq 1$, we may choose a vertex $z\in \Gamma_1^1(z_1z_2)$. Let $Z^2_0=\{z_1,z_2,z\}\subseteq Z^1_1$ and let $\mf{p}_2=\mf{p}_1[Z^2_0;1,r_2]$. Then $\mf{p}_2$ is a $(Z_0^2,d,r_2)$-phantom in $G$, and so $Z^2_0$ is a $3$-clique in $G$ which satisfies \ref{thm:twtokalei}. This complete the proof of Theorem~\ref{thm:twtokalei}.
\end{proof}   

\section{Acknowledgement}

Our thanks to Sophie Spirkl for insightful discussions around Conjecture~\ref{conj:getgo}, and to Bogdan Alecu, Maria Chudnovsky and Sophie Spirkl for hours of joint work on related topics.

\bibliographystyle{abbrv}
	\bibliography{ref}  

\end{document}